\documentclass{amsart}
\usepackage{amssymb}
\usepackage{a4}
\usepackage[all]{xy}
\usepackage{amsmath}
\usepackage{graphics}
\usepackage{stmaryrd}
\usepackage{color}
\usepackage{dsfont}
\usepackage{tikz, booktabs}
\CompileMatrices
\newtheorem*{introthm}{Theorem}

\newtheorem{theorem}{Theorem}[section]
\newtheorem{lemma}[theorem]{Lemma}
\newtheorem{proposition}[theorem]{Proposition}
\newtheorem{corollary}[theorem]{Corollary}
\theoremstyle{definition}

\newtheorem{example}[theorem]{Example}

\newtheorem{construction}[theorem]{Construction}

\newtheorem{remark}[theorem]{Remark}

\def\div{{\rm div}}

\def\Eff{{\rm Eff}}
\def\Mov{{\rm Mov}}
\def\Nef{{\rm Nef}}
\def\SAmple{{\rm SAmple}}

\def\quot{/\!\!/}
\def\mal{\! \cdot \!}

\def\rq#1{\widehat{#1}}
\def\t#1{\widetilde{#1}}
\def\b#1{\overline{#1}}
\def\bangle#1{\langle #1 \rangle}

\def\CC{{\mathbb C}}
\def\KK{{\mathbb K}}
\def\TT{{\mathbb T}}
\def\ZZ{{\mathbb Z}}

\def\NN{{\mathbb N}}
\def\QQ{{\mathbb Q}}
\def\PP{{\mathbb P}}

\def\FF{{\mathbb F}}

\def\WDiv{\operatorname{WDiv}}

\def\id{{\rm id}}

\def\Aut{\operatorname{Aut}}

\def\Cl{\operatorname{Cl}}

\def\CaDiv{\operatorname{CaDiv}}

\def\Pic{\operatorname{Pic}}

\def\Hom{{\rm Hom}}

\def\Spec{{\rm Spec}}

\def\cone{{\rm cone}}

\def\pr{{\rm pr}}

\def\Aut{{\rm Aut}}
\def\rk{{\rm rk}\,}
\def\Bl{{\rm Bl}}

\newcounter{itemnumber}

\begin{document}

\title{On Cox rings of K3-surfaces}

\author{Michela Artebani}
\email{martebani@udec.cl}
\address{
Departamento de Matem\'atica, 
Universidad de Concepci\'on, 
Casilla 160-C, 
Concepci\'on, Chile}

\author{J\"urgen Hausen}
\email{juergen.hausen@uni-tuebingen.de}
\address{
Mathematisches Institut,
Universit\"at T\"ubingen,
Auf der Morgenstelle~10,
72076 T\"ubingen,
Germany}

\author{Antonio Laface}
\email{alaface@udec.cl}
\address{
Departamento de Matem\'atica, 
Universidad de Concepci\'on, 
Casilla 160-C, 
Concepci\'on, Chile}

\thanks{The first author has been partially 
supported by Proyecto FONDECYT Regular 2009, N. 1090069.
The third author has been partially supported 
by Proyecto FONDECYT Regular 2008, N.   1080403.}

\begin{abstract}
We study Cox rings of K3-surfaces.
A first result is that a K3-surface
has a finitely generated Cox ring if 
and only if its effective cone is 
rational polyhedral.
Moreover, we investigate degrees of 
generators and relations for 
Cox rings of K3-surfaces of Picard 
number two,
and explicitly compute the Cox rings
of generic K3-surfaces with a 
non-symplectic involution that
have Picard number 
2 to 5 or occur as double covers of 
del Pezzo surfaces.
\end{abstract}

\maketitle

\section{Introduction}

The Cox ring $\mathcal{R}(X)$ of a normal 
complete algebraic variety $X$ with a 
finitely generated divisor class group 
$\Cl(X)$ is the multigraded algebra 
\begin{eqnarray*}
\mathcal{R}(X)
& := & 
\bigoplus_{\Cl(X)} \Gamma(X, \mathcal{O}_X(D)).
\end{eqnarray*}
For a toric variety $X$, 
the Cox ring $\mathcal{R}(X)$ is a polynomial ring 
and its multigrading can be explicitly determined 
in terms of a defining fan, see~\cite{Cox}.
Moreover, for del Pezzo surfaces $X$ there 
are recent results on generators and 
relations of the Cox ring, 
see~\cite{BaPo}, \cite{De}, 
\cite{lv}, \cite{StiTeVe}, \cite{StuXu},
and~\cite{TeVaVe}.
In the present paper, we investigate  
Cox rings of K3-surfaces $X$, i.e., 
smooth complex projective surfaces $X$  
with $b_1(X) = 0$ that admit a nowhere 
vanishing holomorphic 2-form $\omega_X$.

A first basic problem
is to decide if the Cox ring 
$\mathcal{R}(X)$ 
is finitely generated.
In section~\ref{sec:fingen},
we first discuss this question 
in general, and extend results of Hu 
and Keel on $\QQ$-factorial 
projective varieties~\cite{hk} 
to normal complete ones; 
a consequence is that every normal
complete surface with finitely generated 
Cox ring is $\QQ$-factorial and projective,
see Theorem~\ref{normsurffingen}.
For K3-surfaces, we obtain the following, 
see Theorem~\ref{k3fingen}.

\begin{introthm}
A K3-surface has finitely generated Cox ring
if and only if its cone of effective rational 
divisor classes is polyhedral.
\end{introthm}

The same characterization holds for Enriques 
surfaces, see Theorem~\ref{enriquesfingen}.
The second basic problem is to describe 
the Cox ring $\mathcal{R}(X)$ in terms 
of generators and relations.
We first consider  K3-surfaces~$X$ 
having Picard number $\varrho(X) = 2$,
see section~\ref{sec:rho2}.
In this setting, if the effective cone 
is polyhedral, then it is known that its
generators are of self intersection zero 
or minus two, see section~\ref{sec:fingen} 
for this and some more background.
For the case that both generators 
are of self-intersection zero, 
we obtain the following, see 
Proposition~\ref{genpic2geneff} and 
Theorem~\ref{gen-2}.

\goodbreak

\begin{introthm}
Let $X$ be a K3-surface with  
$\Cl(X) \cong \ZZ w_1 \oplus \ZZ w_2$,
where $w_1$, $w_2$ are effective,
and intersection form given by
$w_1^2 = w_2^2 = 0$ and 
$w_1 \mal w_2 = k \ge 3$. 
\begin{enumerate}
\item
The effective cone of $X$ is 
generated by $w_1$ and $w_2$
and it coincides with 
the semiample cone of $X$.
\item
The Cox ring $\mathcal{R}(X)$ 
is generated in degrees $w_1$, $w_2$,
$w_1 + w_2$, and one has
$$ 
\dim\left(\mathcal{R}(X)_{w_i}\right)
\ = \ 
2,
\qquad
\dim\left(\mathcal{R}(X)_{w_1+w_2}\right)
\ = \
k + 2.
$$
Moreover, every minimal system of generators 
of $\mathcal{R}(X)$ has $k+2$ members.
\item
For $k = 3$, the 
Cox ring $\mathcal{R}(X)$ is of the form 
$\CC[T_1, \ldots, T_5]/ \bangle{f}$ and the
degrees of the generators and the relation 
are given by 
\begin{eqnarray*} 
\deg(T_1) = \deg(T_2) = w_1,
& \quad &
\deg(T_4) = \deg(T_5) = w_2,
\\[.5ex]
\deg(T_3) = w_1+w_2,
& \quad &
\deg(f) = 3w_1+3w_2.
\end{eqnarray*}
\item
For $k \ge 4$, 
any minimal ideal $\mathcal{I}(X)$ of relations of 
$\mathcal{R}(X)$ 
is generated in degree $2w_1+2w_2$, and we have 
\begin{eqnarray*}
\dim\left(\mathcal{I}(X)_{2w_1+2w_2}\right)
& = & 
\frac{k(k-3)}{2}.
\end{eqnarray*}
\end{enumerate}
\end{introthm}

The statements on the generators are directly 
obtained, 
and for the relations, we use the techniques
developped in~\cite{lv}.
When at least one of the generators of the 
effective cone is a $(-2)$-curve, 
then the semiample cone is a proper subset
of the effective cone. 
We show that in this case the number of 
degrees needed to generate the Cox ring 
can be arbitrarily big,
Propositions~\ref{gens-0-2} and~\ref{gens-2-2} 
give a lower bound for this number 
in terms of the intersection 
form of $\Cl(X)$.

For the K3-surfaces $X$ with Picard number 
$\varrho(X)\geq 3$, we use a different 
approach. 
Many K3-surfaces $X$ with $\varrho(X)\geq 3$
and polyhedral effective cone 
admit a non-symplectic involution,
i.e., an automorphism $\sigma \colon X \to X$
of order two with $\sigma^*\omega_X \neq \omega_X$.
The associated quotient map $\pi \colon X \to Y$ 
is a double cover.
If it is unramified then $Y := X / \bangle{\sigma}$
is an Enriques surface, 
otherwise $Y$ is a smooth rational surface. 
In the latter case, one may 
use known results and techniques to 
obtain the Cox ring of $Y$.

This observation suggests to study the 
behaviour of Cox rings under double 
coverings $\pi \colon X \to Y$.
As it may be of independent interest, 
we consider in section~\ref{sec:abcov}
more general, e.g., cyclic,
coverings $\pi \colon X \to Y$
of arbitray normal varieties $X$ and~$Y$.
We relate finite generation of the Cox 
rings of $X$ and $Y$ to each other, see 
Proposition~\ref{prop:abcovfingen},
and Proposition~\ref{prop:abcov}
provides generators and relations for 
the Cox ring of $X$ in terms of $\pi$ 
and the Cox ring of $Y$ for the case 
that $\pi$ induces 
an isomorphism on the level of divisor 
class groups.
This enables us to compute Cox rings of 
K3-surfaces that are general double covers of 
$\FF_0$ or of del Pezzo surfaces.

Besides $\FF_0$ and the del Pezzo surfaces, 
other rational surfaces 
$Y = X / \bangle{\sigma}$
can occur.
For $2 \le \varrho(X) \le 5$, these turn out
to be blow ups of the fourth Hirzebruch surface
$\FF_4$ in at most three general points, 
and we are in this setting if and only 
if the branch divisor of the covering 
$\pi \colon X \to Y$ 
has two components.
Then, in order to determine the Cox ring 
of $X$, we have to solve two problems.
Firstly, the computation of the Cox ring
of $Y$. While blowing up one or two points 
gives a toric surface, the blow up
of $\FF_4$ in three general points is 
non-toric; we compute
its Cox ring in section~\ref{sec:coxblow} 
using the technique of toric ambient 
modifications developped in~\cite{Ha2}.
The second problem is that $\pi \colon X \to Y$ 
induces no longer an isomorphism on the 
divisor class groups. Here,
Proposition~\ref{f4double} provides a 
general result. 

Putting all together,
we obtain the following results
in the case of Picard number 
$2 \le \varrho(X) \le 5$, see
Propositions~\ref{rhoX2} to~\ref{rhoX5}
for the complete statements.

\goodbreak

\begin{introthm}
Let $X$ be a generic K3-surface with a 
non-symplectic involution and associated 
double cover $X \to Y$ and Picard number
$2 \le \varrho(X) \le 5$. 
Then the Cox ring $\mathcal{R}(X)$ 
is given as follows.
\begin{enumerate}
\item
For $\varrho(X) = 2$ one has
$\mathcal{R}(X)
 = 
\CC[T_1,\ldots, T_5]/\bangle{T_5^2 - f}
$
and the degree of $T_i$ is the $i$-th column 
of 
\begin{eqnarray*}
{\tiny
\left[
\begin{array}{rrrrr}
1 & 0 & 1 & 0 & 2
\\
0 & 1 & 0 & 1 & 2
\end{array}
\right]
}
& & 
\text{if} \quad Y \ = \ \FF_0,
\\[.5ex]
{\tiny
\left[
\begin{array}{rrrrr}
1 & 0 & -1 & -1 & -1
\\
0 & 1 & 1 & 1 & 3
\end{array}
\right]
}
& & 
\text{if} \quad Y \ = \ \FF_1,
\\[.5ex]
{\tiny
\left[
\begin{array}{rrrrr}
1 & 0 & 2 & 0 & 3
\\
0 & 1 & 4 & 1 & 6
\end{array}
\right]
}
& & 
\text{if} \quad Y \ = \ \FF_4.
\end{eqnarray*}
\item
For $\varrho(X) = 3$ one has
$\mathcal{R}(X)
 = 
\CC[T_1,\ldots, T_6]/\bangle{T_6^2 - f}
$
and the degree of $T_i$ is the 
$i$-th column of
\begin{eqnarray*}
{\tiny
\left[
\begin{array}{rrrrrr}
1 & 0 & 0 & 1 & 0 & 2
\\
0 & 1 & 0 & 0 & 1 & 2
\\
0 & 0 & 1 & 1 & 1 & 3
\end{array}
\right]
}
& &  \text{if} \quad Y \ = \ \Bl_1(\FF_0),
\\[.5ex]
{\tiny
\left[
\begin{array}{rrrrrr}
1 & 0 & 0 & 2 & 0 & 3
\\
0 & 1 & 0 & 1 & -1 & 1
\\
0 & 0 & 1 & 3 & 1 & 5
\end{array}
\right]
}
& & \text{if} \quad Y \ = \ \Bl_1(\FF_4).
\end{eqnarray*}
\item
For $\varrho(X) = 4$ one has
$\mathcal{R}(X)
 = 
\CC[T_1,\ldots, T_7]/\bangle{T_7^2 - f}
$
and the degree of $T_i$ is the 
$i$-th column of 
\begin{eqnarray*}
{\tiny
\left[
\begin{array}{rrrrrrr}
1 & 0 & 0 & 0 & 1 & 0 & 2
\\
0 & 1 & 0 & 0 & 0 & 1 & 2
\\
0 & 0 & 1 & 0 & 1 & 1 & 3
\\
0 & 0 & 0 & 1 & -1 & -1 & -1 
\end{array}
\right]
}
& &   
\text{if} \quad Y \ = \ \Bl_2(\FF_0),
\\[.5ex]
{\tiny
\left[
\begin{array}{rrrrrrr}
1 & 0 & 0 & 0 & 2 & 0 & 3
\\
0 & 1 & 0 & 0 & 3 & 1 & 5
\\
0 & 0 & 1 & 0 & 1 & -1 & 1
\\
0 & 0 & 0 & 1 & 2 & 1 & 4
\end{array}
\right]
}
& &   
\text{if} \quad Y \ = \ \Bl_2(\FF_4).
\end{eqnarray*}
\item
For $\varrho(X)= 5$ one has the following two cases.
\begin{enumerate}
\item
The surface $Y$ is the blow up of $\FF_0$ 
at three general points.
Then the Cox ring $\mathcal{R}(X)$ of $X$ is
\begin{eqnarray*}
\CC[T_1,\ldots, T_{11}]/\bangle{f_1, \ldots, f_5, \, T_{11}^2 - g},
\end{eqnarray*}
where $f_1, \ldots, f_5, g \in \CC[T_1,\ldots, T_{10}]$ 
and $f_1, \ldots, f_5$ are the Pl\"ucker relations 
in $T_1, \ldots, T_{10}$.
The degree of $T_i$ is the $i$-th column of
\begin{eqnarray*}
{\tiny
\left[
\begin{array}{rrrrrrrrrrr}
0 & 0 & 0 & 0 & 1 & 1 & 1 & 1 & 1 & 1 &    -3 
\\
1 & 0 & 0 & 0 & -1 & -1 & -1 & 0 & 0 & 0 &  1
\\
0 & 1 & 0 & 0 & -1 & 0 & 0 & -1 & -1 & 0 &  1
\\
0 & 0 & 1 & 0 & 0 & -1 & 0 & -1  & 0 & -1 & 1
\\
0 & 0 & 0 & 1 & 0 & 0 & -1 & 0 & -1 & -1 & 1
\end{array}
\right]
}
\end{eqnarray*}
\item
The surface $Y$ is the blow up of $\FF_4$ 
at three general points.
Then the Cox ring $\mathcal{R}(X)$ of $X$ is
\begin{eqnarray*}
\CC[T_1,\ldots, T_9]/\bangle{T_2T_5 + T_4T_6+T_7T_8, \, T_9^2 - f},
\end{eqnarray*}
where $f \in \CC[T_1, \ldots, T_8]$ is a 
prime polynomial and
the degree of $T_i \in \mathcal{R}(X)$ is
the $i$-th column of
\begin{eqnarray*}
{\tiny
\left[
\begin{array}{rrrrrrrrr}
1 & 0 & 0 & 0 & 0 & 0 & -2 & 2 & 1
\\
0 & 1 & 0 & 0 & 0 & 1 & -2 & 3  & 4
\\
0 & 0 & 1 & 0 & 0 & -1 & -1 & 1  & 0
\\
0 & 0 & 0 & 1 & 0 & 1 & -1 & 2 & 4
\\
0 & 0 & 0 & 0 & 1 & 0 & 1 & -1 & 1
\end{array}
\right]
}
\end{eqnarray*}
\end{enumerate}
\end{enumerate}
\end{introthm}

If $Y = X/\bangle{\sigma}$ is a del Pezzo surface, 
then the Cox ring of $Y$ is known, and 
we obtain the following for the Cox ring 
of $X$, see Proposition~\ref{doubledelp}.

\begin{introthm}
Let $X$ be a generic K3-surface 
with a non-symplectic involution,
associated double cover 
$\pi \colon X \to Y$
and intersection form
$U(2)\oplus A_1^{k-2}$, 
where $5 \le k \le 9$.
Then $Y$ is a del Pezzo surface of 
Picard number $k$
and
\begin{enumerate}
\item
the Cox ring $\mathcal{R}(X)$ is generated by the
pull-backs of the $(-1)$-curves of $Y$, 
the section $T$ defining the ramification divisor
and, for $k =9$, 
the pull-back of an irreducible section
of $H^0(Y,- K_Y)$,
\item
the ideal of relations of $\mathcal{R}(X)$ 
is generated by quadratic
relations of degree $\pi^*(D)$, 
where $D^2=0$ and $D \mal K_Y=-2$,
and the relation $T^2-f$ in degree 
$-2\pi^*(K_Y)$, where $f$ is the pullback
of the canonical section of the branch divisor.
\end{enumerate}
\end{introthm}

We would like to thank the referee for several 
comments which helped to clarify the exposition.

\section{Finite generation of the Cox ring}
\label{sec:fingen}

The first result of this section is 
Theorem~\ref{fingenchar}, which 
characterizes finite generation 
of the Cox ring for normal complete varieties
in a similar way as  Hu and Keel
did for $\QQ$-factorial projective varieties 
in~\cite{hk}. 
In the case of surfaces, we obtain statements
extending and complementing results of Galindo 
and Monserrat~\cite{GalMo}; a consequence is 
that every normal complete surface with finitely
generated Cox ring is projective and $\QQ$-factorial,
see Theorem~\ref{normsurffingen}.
Our main application is Theorem~\ref{k3fingen},
which characterizes finite generation
of the Cox ring of a K3-surface.

Let $X$ be a normal complete algebraic variety
defined over some algebraically 
closed field~$\KK$ of characteristic zero
and assume that the divisor class group 
$\Cl(X)$ is finitely generated.
The {\em Cox ring\/} of $X$ is the 
ring $\mathcal{R}(X)$ of global sections 
of a sheaf $\mathcal{R}$ of $\Cl(X)$-graded 
$\KK$-algebras.
We briefly recall the definition for the 
case of that $\Cl(X)$ is free and refer 
to~\cite{Ha2} for the case of torsion:
take a subgroup $K \subseteq \WDiv(X)$ of the 
group of Weil divisors such that the canonical 
morphism $K \to \Cl(X)$ is an 
isomorphism and set
$$
\mathcal{R}(X)
\quad := \quad
\Gamma(X,\mathcal{R}),
\qquad
\qquad
\mathcal{R}
\quad := \quad
\bigoplus_{D \in K} \mathcal{O}_X(D),
$$
where multiplication is defined via multiplying 
homogeneous sections in the field of rational
functions of $X$.
Up to isomorphy this definition does not 
depend on the choice of the group 
$K \subseteq \WDiv(X)$.
We will also identify $K$ with 
$\Cl(X)$. So, for $w \in \Cl(X)$ 
represented by $D \in K$, the 
homogeneous component $\mathcal{R}(X)_w$ 
is just $H^0(X,\mathcal{O}_X(D))$.

If the divisor class group
$\Cl(X)$ is free, then the Cox ring 
$\mathcal{R}(X)$  admits
unique factorization, see~\cite{BeHa1}.
If $\Cl(X)$ has torsion, then the unique
factorization property is replaced by 
a graded version: every nontrivial homogeneous
nonunit is a product of homogeneous primes,
where the latter refers to nontrivial homogeneous 
nonunits $f$ such that $f \vert gh$ with $g,h$ 
homogeneous implies $f \vert g$
or $f \vert h$, see~\cite{Ha2}.
If the Cox ring  $\mathcal{R}(X)$ 
is finitely generated as a $\KK$-algebra,
then one may define the 
{\em total coordinate space\/} 
$\b{X} = \Spec \, \mathcal{R}(X)$ and 
realize $X$ as the quotient of an open 
subset $\rq{X} = \Spec_X \mathcal{R}$ 
of $\b{X}$ by the action of the 
diagonalizable group $\Spec \, \Cl(X)$ 
defined by the 
$\Cl(X)$-grading of the sheaf $\mathcal{R}$
of $\mathcal{O}_X$-algebras.
For smooth $X$, the map $\rq{X} \to X$ 
is also known as the {\em universal torsor\/}
of $X$.

Let 
$\Cl_{\QQ}(X) = \Cl(X) \otimes_{\ZZ} \QQ$
denote the rational divisor class group
of $X$. 
A first step is to give descriptions 
of the cones
$\Eff(X) \subseteq \Cl_{\QQ}(X)$ 
of effective classes and 
$\Mov(X) \subseteq \Cl_{\QQ}(X)$
of movable classes, i.e.,~classes 
having a stable base locus of codimension 
at least two. We call a cone in 
a rational vector space $V$ polyhedral 
if it is the positive hull
$\cone(v_1, \ldots, v_r)$ of
finitely many vectors $v_i \in V$. 
The following statement generalizes
part of~\cite[Prop.~4.1]{Ha2}.

\begin{proposition}
\label{prop:effmov}
Let $X$ be a normal complete variety
with finitely generated Cox ring 
$\mathcal{R}(X)$.
Then the cones of effective and 
movable divisor classes in $\Cl_{\QQ}(X)$
are polyhedral. 
Moreover, if $f_1, \ldots, f_r \in \mathcal{R}(X)$
is any system of pairwise non-associated 
homogeneous prime generators, then one has
\begin{eqnarray*}
\Eff(X)
& = &
\cone(\deg(f_i); \; i = 1, \ldots, r),
\\
\Mov(X)
& = &
\bigcap_{i=1}^r
\cone(\deg(f_j); \; j \ne i).
\end{eqnarray*}
\end{proposition}

For the proof and also for later use, we 
have to fix some notation. 
On a normal variety $X$, let 
a class $w \in \Cl(X)$ 
be represented by a divisor
$D \in \WDiv(X)$.
Then, as usual, we write 
$$ 
H^i(D) 
\ := \ 
H^i(X,D)
\ := \ 
H^i(X,\mathcal{O}_X(D)),
\quad
h^i(w) 
\ := \
h^i(D)
\ := \ 
\dim_\KK(H^i(D)).
$$

\goodbreak

\begin{lemma}
\label{lem:pol2mov}
Let $X$ be a normal complete variety  
with $\Cl(X)$ finitely generated
and let $w \in \Cl(X)$ be effective.
Then the following two statements
are equivalent.
\begin{enumerate}
\item
The stable base locus of the class 
$w \in \Cl(X)$ contains a divisor.
\item
There exist a class $w_0 \in \Cl(X)$
with the following properties
\begin{itemize}
\item
the class $w_0$ generates an extremal 
ray of $\Eff(X)$ and $h^0(nw_0) = 1$ 
holds for any $n \in \NN$,
\item
there is an $f_0 \in \mathcal{R}(X)_{w_0}$ 
such that for any $n \in \NN$ and 
$f \in \mathcal{R}(X)_{nw}$ one has 
$f = f'f_0$ with some
$f' \in \mathcal{R}(X)_{nw-w_0}$. 
\end{itemize}
\end{enumerate}
\end{lemma}

\begin{proof}
The implication ``(ii)$\Rightarrow$(i)''
is obvious.
So, assume that~(i) holds.
The class $w \in \Cl(X)$ is represented 
by some non-negative divisor $D$.
Let $D_0$ be a prime component of $D$,
which occurs in the fixed part of any 
positive multiple $nD$,
and let $w_0 \in \Cl(X)$ be the class of 
$D_0$.
Then the canonical section of $D_0$ 
defines an element $f_0 \in \mathcal{R}(X)_{w_0}$,
which divides any $f \in \mathcal{R}(X)_{nw}$. 
Note that $h^0(nD_0) = 1$ holds for 
every $n \in \ZZ_{\ge 0}$,
because otherwise
$H^0(na_0D_0)\subseteq \mathcal{R}(X)_{nw}$,
where $a_0 > 0$ is the multiplicity of
$D_0$ in $D$,  
would provide enough sections to move 
$na_0D_0$.  
Moreover, $\cone(w_0)$ is an extremal 
ray of $\Eff(X)$, because otherwise 
we had $nD_0 \sim D_1+D_2$
with some $n \in \ZZ_{\ge 0}$
and non-negative divisors $D_1$, $D_2$,
none of which is a multiple of
$D_0$; this contradicts 
$h^0(nD_0) = 1$.
\end{proof}

\begin{proof}[Proof of Proposition~\ref{prop:effmov}]
Only for the description of the moving cone,
there is something to show.
For this, set $w_i := \deg(f_i)$ and note that the 
extremal rays of the effective cone occur among the 
$\QQ_{\ge 0} \mal w_i$.
By suitably renumbering we achieve that
the indices $1 \le i \le d$ are precisely those 
with $h^0(nw_i) \le 1$ for all $n \in \NN$.

Let $w \in \Mov(X)$.
Then Lemma~\ref{lem:pol2mov} tells us 
that for any $i = 1, \ldots, d$, 
there must be a monomial of the form
$\prod_{j \ne i} f_j^{n_j}$
in some $\mathcal{R}(X)_{nw}$.
Consequently, $w$ lies in the cone
of the right hand side.
Conversely, consider an element  $w$ of 
the cone of the right hand side.
Then, for every $i = 1, \ldots, d$,
a product $\prod_{j \ne i} f_j^{n_j}$
belongs to some $\mathcal{R}(X)_{nw}$.
Hence none of the $f_1, \ldots, f_d$
divides all elements of $\mathcal{R}(X)_{nw}$.
Again by Lemma~\ref{lem:pol2mov},
we conclude $w \in \Mov(X)$.
\end{proof}

We characterize finite 
generation of the Cox ring. 
By $\SAmple(X) \subseteq \Cl_{\QQ}(X)$
we denote the cone of semiample divisor 
classes of a variety $X$, i.e., classes 
having a basepoint free positive multiple.
Moreover, by a small birational map $X \to Y$, 
we mean a rational map that defines an 
isomorphism $U \to V$ of open subsets 
$U \subseteq X$ and $V \subseteq Y$ such 
that the respective complements $X \setminus U$ 
and $Y \setminus V$ are of codimension at least
two.

\begin{theorem}
\label{fingenchar}
Let $X$ be a normal complete variety 
with finitely generated divisor class group. 
Then the following statements are equivalent.
\begin{enumerate}
\item
The Cox ring $\mathcal{R}(X)$ is finitely generated.
\item
The effective cone $\Eff(X) \subseteq \Cl_{\QQ}(X)$ 
is polyhedral and there are 
small birational maps $\pi_i \colon X \to X_i$, 
where $i = 1, \ldots, r$,
such that each semiample cone 
$\SAmple(X_i) \subseteq \Cl_{\QQ}(X)$ is polyhedral and 
one has
\begin{eqnarray*}
\Mov(X)
& = & 
\pi_1^*(\SAmple(X_1)) \ \cup \ \ldots \ \cup \ \pi_r^*(\SAmple(X_r)).
\end{eqnarray*}
\end{enumerate}
Moreover, if one of these two statements holds, 
then there is a small birational map $X \to X'$
with a $\QQ$-factorial projective variety $X'$.
\end{theorem}

In the proof we use that the moving cone 
of any normal complete variety is of full 
dimension; we are grateful to Jenia~Tevelev
for providing us with the following 
statement and proof.

\begin{lemma}
\label{movfull}
Let $X$ be a normal complete variety with $\Cl(X)$ 
finitely generated.
Then the moving cone $\Mov(X)$ is of full 
dimension in the rational divisor class group
$\Cl_{\QQ}(X)$.
\end{lemma}

\begin{proof}
Using Chow's Lemma and resolution of 
singularities, we obtain a birational 
morphism $\pi \colon X' \to X$ with a 
smooth projective variety $X'$.
Let $D_1, \ldots, D_r \in \WDiv(X)$ be 
prime divisors generating $\Cl(X)$,
and consider their proper transforms
$D_1', \ldots, D_r' \in \WDiv(X')$.
Moreover, let $E' \in \CaDiv(X')$ be 
very ample such that all $E'+D_i'$ 
are also very ample, and denote by
$E \in \WDiv(X)$ its push-forward.
Then the classes $E$ and $E+D_i$ generate
a fulldimensional cone 
$\tau \subseteq \Cl_{\QQ}(X)$ 
and, since $E'$ and the $E'+D_i'$ are
movable, we have $\tau \subseteq \Mov(X)$.  
\end{proof}

\begin{proof}[Proof of Theorem~\ref{fingenchar}]
Suppose that~(i) holds.
Then Proposition~\ref{prop:effmov} 
tells us that $\Eff(X)$ is polyhedral.
Let $\mathfrak{F} = (f_1, \ldots, f_r)$ be a 
system of pairwise nonassociated homogeneous 
prime generators of $R := \mathcal{R}(X)$
and set $w_i := \deg(f_i)$.

By~\cite[Prop.~2.2]{Ha2},
the group $H = \Spec \, \KK[\Cl(X)]$ 
acts freely on an open subset 
$W \subseteq \b{X}$ of 
$\b{X} = \Spec \, \mathcal{R}(X)$
such that $\b{X} \setminus W$ 
is of codimension at least two in $\b{X}$.
Thus, we can choose a point $z \in W$ with 
$f_i(z) = 0$ and $f_j(z) \ne 0$ for $j \ne i$.
Consequently, the weights $w_j$, where $j \ne i$,
generate $\Cl(X)$ and hence the system
of generators $\mathfrak{F}$ 
is admissible in the sense 
of~\cite[Def.~3.4]{Ha2}.
Moreover, by Lemma~\ref{movfull}, 
the moving cone
of $X$ is of full dimension, and by 
Proposition~\ref{prop:effmov},
it is given as
\begin{eqnarray*}
\Mov(X)
& = & 
\bigcap_{i=1}^r \cone(w_j; \; j \ne i).
\end{eqnarray*}
Thus, we are in the setting of~\cite[Cor.~4.3]{Ha2}.
That means that $\Mov(X)$ is a union of 
fulldimensional GIT-chambers 
$\lambda_1, \ldots, \lambda_r$,
the relative interiors of which 
are contained in the relative interior
of $\Mov(X)$ and the associated projective 
varieties $X_i := \rq{X}_i \quot H$,
where $\rq{X}_i := \b{X}^{ss}(\lambda_i)$,
are $\QQ$-factorial,
have $\mathcal{R}(X)$ as their Cox ring 
and $\lambda_i$ as their semiample cone.

Moreover, if $q \colon \rq{X} \to X$ and 
$q_i \colon \rq{X}_i \to X_i$ denote the 
associated universal torsors, then 
the desired small birational maps 
$\pi_i \colon X \to X_i$ are obtained as follows.
Let $X' \subseteq X$ and $X_i' \subseteq X_i$ be
the respective sets of smooth points.
Then, by~\cite[Prop.~2.2]{Ha2},
the sets $q^{-1}(X')$ and $q^{-1}(X_i')$ 
have a small complement in $\b{X}$ and 
thus we obtain open embeddings with 
a small complement
$$ 
\xymatrix{
X
&
(q^{-1}(X') \cap q^{-1}(X_i')) \quot H
\ar[l]
\ar[r]
&
X_i.
}
$$

Now suppose that~(ii) holds.
Let $w_1, \ldots, w_d \in \Eff(X)$ be those 
primitive generators of extremal rays of $\Eff(X)$ 
that satisfy $h^0(nw_i) \le 1$ for any $n \in \ZZ_{\ge 0}$
and fix $0 \ne f_i \in  \mathcal{R}(X)_{n_iw_i}$ with 
$n_i$ minimal. Then we have 
\begin{eqnarray*}
\bigoplus_{n \in \ZZ_{\ge 0}} \mathcal{R}(X)_{nw_i}
& = & 
\KK[f_i].
\end{eqnarray*}
Set $\lambda_i := \pi_i^* (\SAmple(X_i))$.
Then, by Gordon's Lemma and~\cite[Lemma~1.8]{hk},
we have another finitely generated subalgebra 
of the Cox ring, namely
$$
\mathcal{S}(X) 
\ := \
\bigoplus_{w \in \Mov(X)} \mathcal{R}(X)_w
\ = \ 
\bigoplus_{i=1}^r
\left(\bigoplus_{w \in \lambda_i} \mathcal{R}(X)_w\right).
$$
We show that $\mathcal{R}(X)$ is generated 
by $\mathcal{S}(X)$ and the 
$f_i \in \mathcal{R}(X)_{n_iw_i}$. 
Consider any $0 \ne f \in \mathcal{R}(X)_w$ 
with $w \not\in \Mov(X)$.
Then, by Lemma~\ref{lem:pol2mov}, we have 
$f = f^{(1)}f_i$ for some $1 \le i \le d$
and some $f^{(1)} \in \mathcal{R}(X)$ 
homogeneous of degree $w(1) := w-n_iw_i$.
If $w(1)  \not \in \Mov(X)$
holds, then we repeat this procedure 
with $f^{(1)}$ and 
obtain $f = f^{(2)} f_if_j$ with 
$f^{(2)}$ homogeneous of degree $w(2)$.
At some point, we must end with 
$w(n) = \deg(f^{(n)}) \in \Mov(X)$,
because otherwise the sequence of the 
$w(n)$'s would leave the 
effective cone.
\end{proof}

\goodbreak

\begin{theorem}
\label{normsurffingen}
Let $X$ be a normal complete surface with 
finitely generated divisor class group 
$\Cl(X)$. 
Then the following statements are
equivalent.
\begin{enumerate}
\item
The Cox ring $\mathcal{R}(X)$ is finitely generated.
\item 
The effective cone $\Eff(X) \subseteq \Cl_\QQ(X)$ 
and the moving cone $\Mov(X) \subseteq \Cl_\QQ(X)$ 
are polyhedral and $\Mov(X) = \SAmple(X)$ holds.
\end{enumerate}
Moreover, if one of these two statements holds, then 
the surface $X$ is $\QQ$-factorial and projective.
\end{theorem}

\begin{proof}
We verify the implication ``(i)$\Rightarrow$(ii)''.
By Proposition~\ref{prop:effmov},
we only have to show that 
the moving cone coincides with the 
semiample cone.
Clearly, we have 
$\SAmple(X) \subseteq \Mov(X)$.
Suppose that
$\SAmple(X) \neq \Mov(X)$
holds.
Then $\Mov(X)$ is properly subdivided
into GIT-chambers, see~\cite[Cor.~4.3]{Ha2}.
In particular, we find two chambers
$\lambda'$ and $\lambda$ both intersecting 
the relative interior of $\Mov(X)$ such 
that $\lambda'$ is a proper face of $\lambda$.
The associated GIT-quotients 
$Y'$ and $Y$ of the total coordinate 
space $\b{X}$
have $\lambda'$ and $\lambda$ as their respective 
semiample cones.
Moreover, the inclusion $\lambda' \subseteq \lambda$
gives rise to a proper morphism 
$Y \to Y'$, which is an isomorphism 
in codimension one.
As $Y$ and $Y'$ are normal surfaces, we 
obtain $Y \cong Y'$, which contradicts 
the fact that the semiample cones of $Y$ and $Y'$ 
are of different dimension.

The verification of ``(ii)$\Rightarrow$(i)'' 
runs as in the preceding proof; this time 
one uses the finitely generated subalgebra 
$$
\mathcal{S}(X) 
\ := \
\bigoplus_{w \in \Mov(X)} \mathcal{R}(X)_w
\ = \ 
\bigoplus_{w \in \SAmple(X)} \mathcal{R}(X)_w.
$$
Moreover, by Theorem~\ref{fingenchar}, there is 
a small birational map $X \to X'$ with $X'$ projective 
and $\QQ$-factorial. 
As $X$ and $X'$ are complete surfaces, this map already 
defines an isomorphism.
\end{proof}

In the case of a $\QQ$-factorial surface $X$, 
we obtain the following simpler
characterization involving the cone 
$\Nef(X) \subseteq \Cl_{\QQ}(X)$ 
of numerically effective divisor classes; 
note that the implication ``(ii)$\Rightarrow$(i)'' 
was obtained for smooth surfaces 
in~\cite[Cor.~1]{GalMo}.

\goodbreak

\begin{corollary}
\label{smoothfingen}
Let $X$ be a $\QQ$-factorial projective surface with 
finitely generated divisor class group 
$\Cl(X)$. Then the following statements are
equivalent.
\begin{enumerate}
\item
The Cox ring $\mathcal{R}(X)$ is finitely generated.
\item
The effective cone $\Eff(X) \subseteq \Cl_\QQ(X)$ 
is polyhedral and $\Nef(X) = \SAmple(X)$ holds. 
\end{enumerate}
\end{corollary}

\begin{proof}[Proof of Corollary~\ref{smoothfingen}]
If~(i) holds, then we infer from~\cite[Cor.~7.4]{BeHa2}
that the semiample cone and the nef cone of 
$X$ coincide.
Now suppose that~(ii) holds. 
From   
$$
\SAmple(X) 
\ \subseteq \ 
\Mov(X) 
\ \subseteq \ 
\Nef(X)
$$
we then conclude $\Mov(X) = \Nef(X)$. 
Moreover, since $\Eff(X)$ is polyhedral,
$\Nef(X)$ is given by a finite number 
of inequalities and hence is 
also polyhedral.
Thus, we can apply Theorem~\ref{normsurffingen}.
\end{proof}

We turn to K3-surfaces $X$.
Recall that by definition, $X$ is a smooth 
complete complex surface with  $b_1(X) = 0$ and
trivial canonical class.
We always assume a K3-surface $X$ to be {\em algebraic}.
As a sublattice of $H^2(X,\ZZ) \cong \ZZ^{22}$, 
the divisor class group $\Cl(X)$ is 
finitely generated and free. 
In particular, we can define a Cox ring 
$\mathcal{R}(X)$ as above.
Our first result characterizes finite 
generation of $\mathcal{R}(X)$.

\goodbreak

\begin{theorem}
\label{k3fingen}
For any complex algebraic K3-surface~$X$, 
the following statements are equivalent.
\begin{enumerate}
\item
The Cox ring $\mathcal{R}(X)$ is finitely 
generated.
\item
$\Eff(X)$ is polyhedral.
\end{enumerate}
\end{theorem}

\begin{lemma}
\label{negative}
Let $X$ be a K3-surface and $D$ be a non-principal divisor
on $X$.
If we have $h^0(D)=1$, then  $D^2<0$ holds.
\end{lemma}

\begin{proof}
Since the canonical divisor of $X$ 
is principal, Serre's duality theorem 
gives us $h^2(D) = h^0(-D) = 0$.
The Riemann-Roch theorem then yields
$1 \geq D^2/2+2$.
The assertion follows.
\end{proof}

\begin{proof}[Proof of Theorem~\ref{k3fingen}]
Only for ``(ii)$\Rightarrow$(i)'' there is 
something to show.
So, assume that $\Eff(X)$ is 
polyhedral.
Then, by Corollary~\ref{smoothfingen}, 
we have to 
show that for every numerically effective 
divisor $D$ on $X$, 
some positive multiple is semiample.

By a result of Kleiman, the class 
$[D] \in \Cl_{\QQ}(X)$ 
lies in the closure of the cone of 
ample divisor classes. 
Since any ample class is effective
and $\Eff(X)$ as a polyhedral cone is 
closed in $\Cl_{\QQ}(X)$, we obtain 
$[D] \in \Eff(X)$.
Thus, we may assume that $D$ is
non-negative. 

Since $D$ is numerically effective, 
we have $D^2 \ge 0$. 
If $D^2 > 0$ holds, then~\cite[Cor., p.~11]{Ma}
tells us that the linear system~$\vert 3D \vert$ 
is base point free, i.e., that $3D$ is 
semiample.
If we have $D^2 = 0$, then we 
write $D = D_0 + D_1$,
where $D_0$ denotes the fixed part of $D$ 
and $D_0,D_1$ are non-negative.
Then we have
$$ 
0 
\quad = \quad
D^2 
\quad = \quad
D \mal D_0 + D \mal D_1.
$$
Since $D_0$ and $D_1$ are non-negative, 
we conclude $D \mal D_0 = 0$ and 
$D \mal D_1 = 0$.
We show that $D_0=0$ must hold.
Otherwise, using Lemma~\ref{negative}, 
we obtain
$$ 
D_1^2
\quad = \quad
(D - D_0)^2 
\quad < \quad
0.
$$
On the other hand $D_1$ has no fixed 
components. Thus, according to~\cite[Cor.~3.2]{SD},
the divisor $D_1$ is base point free and
hence numerically effective. 
A contradiction.
Thus, we see that $D = D_1$ holds, and 
thus $D$ is semiample.
\end{proof}

\begin{corollary}
Let $X$ be a K3-surface. 
If the cone 
$\Eff(X) \subseteq \Cl_{\QQ}(X)$ 
of effective divisor classes 
is polyhedral,
then also the cone 
$\SAmple(X) \subseteq \Cl_{\QQ}(X)$ 
of semiample divisor classes is 
polyhedral.
\end{corollary}

For Enriques surfaces, i.e., smooth projective 
surfaces $X$ with $q(X) = 0$ and 
$2K_X$ trivial but $K_X$ 
nontrivial,
we obtain the following analogue 
of Theorem~\ref{k3fingen}.

\begin{theorem}
\label{enriquesfingen}
For any Enriques surface~$X$,
the following statements are equivalent.
\begin{enumerate}
\item
The Cox ring $\mathcal{R}(X)$ is finitely
generated.
\item
$\Eff(X)$ is polyhedral.
\end{enumerate}
\end{theorem}

\begin{proof}
Only for ``(ii)$\Rightarrow$(i)'' there is
something to show.
So, assume that $\Eff(X)$ is
polyhedral.
Then, by Corollary~\ref{smoothfingen},
we have to
show that for every given 
numerically effective divisor $D$ 
on $X$,
some positive multiple is semiample.
Since $\Eff(X)$ is polyhedral, we 
obtain $\Nef(X) \subseteq \Eff(X)$
and hence we may assume that $D$ is 
nonnegative.

Let $\pi \colon S \to X$ be the universal
covering.
Then $S$ is a K3-surface and $\pi$ is an 
unramified double covering.
The pullback $\pi^*D$ on $S$ is effective and 
numerically effective.
As in the proof of Theorem~\ref{k3fingen}, we 
see that some positive multiple $\pi^*nD$ 
is semiample. 
From~\cite[Lemma 17.2]{BPV} we infer
\begin{eqnarray*}
 H^0(S,\pi^*nD) 
& = & 
\pi^*H^0(X,nD) + \pi^*H^0(X,nD+ K_X).
\end{eqnarray*}
If $x \in X$ is a base point of $nD$,
then there is a $g\in H^0(X,nD+ K_X)$ such
that $g(x)\neq 0$ holds; otherwise $\pi^{-1}(x)$
would be in the base locus of $\pi^*nD$,
which is a contradiction.
Since $2 K_X$ is trivial, we deduce that $2nD$
has no base points.
\end{proof}

We conclude the section with recalling
some classical statements 
on algebraic K3-surfaces $X$
characterizing the case of 
a polyhedral effective cone 
and thus providing further 
criteria for finite generation of the 
Cox ring.
Consider the lattice $\Cl(X) = \Pic(X)$ 
with the intersection pairing,
denote by $\operatorname{O}(\Cl(X))$ 
the group of its isometries and 
by $W(\Cl(X))$ the Weyl group,
i.e., the subgroup generated by reflections 
with respect to $\delta\in \Cl(X)$ with 
$\delta^2=-2$. 

\begin{theorem}
\label{k3eff}
See~\cite[Theorem 2, Remark 7.2]{k} and~\cite[\S 7, Corollary]{pss}.
For any algebraic K3-surface $X$, 
the following statements are equivalent.
\begin{enumerate}
\item 
The cone $\Eff(X) \subseteq \Cl_\QQ(X)$ 
is polyhedral.
\item 
The set $\operatorname{O}(\Cl(X))/W(\Cl(X))$ 
is finite.
\item 
The automorphism group $\Aut(X)$ is finite.
\end{enumerate}
Moreover, if the Picard number 
is at least three, then~(i) is  
equivalent to the property that $X$ contains
only finitely many smooth rational curves. 
In this case, the classes of such curves 
generate the effective cone.
\end{theorem}

The hyperbolic lattices 
satisfying~(ii) 
have been classified in~\cite{pss} 
and a series of papers by V.V.~Nikulin,
see~\cite{n1}, \cite{n2} and~\cite{n3}. 
In particular it has been proved that 
there are only finitely many of them 
having rank at least three. 
The results of these papers, 
together with Theorem~\ref{k3eff}, 
give the following.

\begin{theorem}
\label{ne}
See~\cite{n1,n2,n3}.
Let $X$ be an algebraic K3-surface with 
Picard number $\varrho(X)$. 
\begin{enumerate}
\item 
Suppose that $\varrho(X)=2$ holds.
Then $\Eff(X)$ is polyhedral if and only 
if $\Cl(X)$ contains a class of
self-intersection $0$ or $-2$.
\item 
Suppose that $\varrho(X)\geq 3$ holds.
Then $\Eff(X)$ is polyhedral 
if and only if $\Cl(X)$ belongs to a 
finite list of hyperbolic lattices. 
The following table gives the number $n$ 
of these lattices for any Picard number:
\begin{center} 
\begin{tabular}{c|ccccccccccc}
$\varrho(X)$& $3$ & $4$ & $5-6$ & $7$ & $8$& $9$ 
& $10$ & $11- 12$& $13-14$ & $15-19$ & $20$   \\
\midrule
$n$ & $27$ &$17$ & $10$ & $9$ & $12$ & $10$ & $9$ & $4$ & $3$ & $1$ & $0$\\
\end{tabular}
\end{center}
\end{enumerate}
 \end{theorem}

The following statement is a consequence of the results 
mentioned above or of~\cite[Theorem~2]{k}.

\begin{proposition}
\label{ko}
Let $X$ be an algebraic K3-surface 
such that $\Eff(X)$ is polyhedral.
Then the generators of $\Eff(X)$
are described in the following table.
\begin{center}
\begin{tabular}{l|ll}
$\varrho(X)$  & $\Eff(X)$ & Type of generators \\[3pt]
\midrule
$1$ & $\QQ_+[H]$ & ample divisor\\[5pt]
$2$ & $\QQ_+[E_1]+\QQ_+[E_2]$ & $(-2)$ or $(0)$-curves\\[5pt]
$3 - 19$ & $\sum \QQ_+[E_i]$ & $(-2)$-curves
\end{tabular}
\end{center}
\end{proposition}

Note that for $\varrho(X)=1$, the Cox ring 
of $X$ coincides with its usual 
homogeneous coordinate ring, 
whose generators have been studied 
in~\cite{SD}.

\section{K3-surfaces of Picard number two}
\label{sec:rho2}

We consider (complex algebraic) 
K3-surfaces $X$ with divisor class 
group $\Cl(X) = \ZZ w_1 \oplus \ZZ w_2$,
where $w_i^2 \in \{0,-2\}$ and, as we may 
assume then, $w_1 \mal w_2 \ge 1$ hold;
recall from~\cite[Cor.~2.9~(i)]{Mo} that 
any even lattice of rank two with 
signature~$(1,1)$ is the Picard lattice 
of an algebraic K3-surface.
According to Theorem~\ref{k3fingen} and the 
characterization of $\Eff(X)$ being
polyhedral provided in Theorem~\ref{ne}, 
such surfaces $X$ have a finitely generated 
Cox ring $\mathcal{R}(X)$.
We investigate the possible degrees of 
generators and relations for $\mathcal{R}(X)$.
An explicit computation of $\mathcal{R}(X)$
for the cases $w_1 \mal w_2=1,2$ is 
given in Section~\ref{sec-double}.
A first observation concerns the effective 
cone.

\begin{proposition}
\label{genpic2geneff}
Let $X$ be a K3-surface with 
$\Cl(X)= \ZZ w_1 \oplus \ZZ w_2$,
where $w_1, w_2$ are effective
such that $w_i^2 \in \{0,-2\}$ 
and $w_1 \mal w_2 \geq 2$ hold.
Then $w_1$ and $w_2$ generate
$\Eff(X)$ as a cone.
\end{proposition}

\begin{proof}
Suppose that 
$\cone(w_1,w_2) \subsetneq \Eff(X)$ 
holds.
Then we may assume that $w_1$ does not 
lie on the boundary of $\Eff(X)$.
Thus, one of the generators
of $\Eff(X)$ is of the form 
$w = aw_1 - bw_2$ for some $a,b \in \NN$,
where $a > 0$.
Proposition~\ref{ko} gives
$$ 
w^2
\ = \ 
a^2 w_1^2 
+ 
b^2 w_2^2
-
2  ab \, w_1 \mal w_2
\ \in \ 
\{0,-2\}.
$$
This can only be realized for $b = 0$,
because we assumed $w_i^2 \in \{0,-2\}$
and $w_1 \mal w_2 \ge 2$.
Thus, $w = aw_1$ holds and,
consequently, $w_1$ lies 
on the boundary of $\Eff(X)$;
a contradiction.
\end{proof}

Note that the assumption of $w_1$ and $w_2$ 
being effective in Proposition~\ref{genpic2geneff} 
can always be achieved: 
Riemann-Roch and $w_i^2 \in \{0,-2\}$ show that 
either $w_i$ or $-w_i$ is effective.

Our next result settles the case $w_i^2 = 0$
and $w_1 \mal w_2 \ge 3$.
In order to state it, we first have to
fix our usage.
Consider any finitely generated 
$\CC$-algebra $R$, graded by a lattice
$K$. 
We say that a system of homogeneous generators
$f_1,\dots,f_r$ of $R$ is {\em minimal\/}
if no $f_i$ can be expressed as a polynomial
in the remaining $f_j$.
Moreover, we say that $R$ 
{\em has a generator\/}
in degree $w \in K$ if any minimal system
of generators for $R$ contains a nontrivial
element of $R_w$.
Given a system $f_1,\dots,f_r$ 
of generators, we have the surjection
$$
\CC[T_1,\ldots,T_r] \ \to \ R,
\qquad
T_i \ \mapsto \ f_i.
$$ 
The {\em ideal of relations\/} 
determined by
$f_1,\dots,f_r$ is the kernel 
$I \subseteq \CC[T_1,\ldots,T_r]$
of this map, it is homogeneous w.r.t.~the 
$K$-grading of $\CC[T_1,\ldots,T_r]$
defined by $\deg(T_i) := \deg(f_i)$.
By a {\em minimal ideal\/} of relations, we 
mean the ideal of relations determined 
by a minimal system of generators.

\begin{theorem}
\label{gen-2}
Let $X$ be a K3-surface with  
$\Cl(X) \cong \ZZ w_1 \oplus \ZZ w_2$,
where $w_1$, $w_2$ are effective,
and intersection form given by
$w_1^2 = w_2^2 = 0$ and 
$w_1 \mal w_2 = k \ge 3$.
\begin{enumerate}
\item
The semiample cone of $X$ coincides with its 
effective cone.
\item
The Cox ring $\mathcal{R}(X)$ 
is generated in degrees $w_1$, $w_2$, 
$w_1 + w_2$, and one has
$$ 
\dim\left(\mathcal{R}(X)_{w_i}\right)
\ = \ 
2,
\qquad
\dim\left(\mathcal{R}(X)_{w_1+w_2}\right)
\ = \
k + 2.
$$
Moreover, any minimal system of generators 
of $\mathcal{R}(X)$ has $k+2$ members.
\item
For $k = 3$, the 
Cox ring $\mathcal{R}(X)$ is of the form 
$\CC[T_1, \ldots, T_5]/ \bangle{f}$ and the
degrees of the generators and the relation 
are given by 
$$ 
\deg(T_1) = \deg(T_2) = w_1,
\qquad
\deg(T_4) = \deg(T_5) = w_2,
$$ 
$$
\deg(T_3) = w_1+w_2,
\qquad
\deg(f) = 3w_1+3w_2.
$$
\item
For $k \ge 4$, 
any minimal ideal $\mathcal{I}(X)$ of relations of 
$\mathcal{R}(X)$ 
is generated in degree $2w_1+2w_2$, and we have 
\begin{eqnarray*}
\dim\left(\mathcal{I}(X)_{2w_1+2w_2}\right)
& = & 
\frac{k(k-3)}{2}.
\end{eqnarray*}
\end{enumerate}
\end{theorem}

Note that for $k=3,4$ the Cox ring 
of $X$ is a complete intersection,
while for $k \ge 5$ this no longer 
holds.
Before giving the proof of the above theorem,
we briefly provide the necessary ingredients.

\begin{lemma}
\label{gens}
Let $X$ be a smooth surface,
assume that $D, D_1, D_2 \in \WDiv(X)$ 
satisfy  $D_1 \cdot D_2 = 0$ and $h^1(D-D_1-D_2) = 0$,
and let $0 \ne f_i \in H^0(D_i)$
such that $\div(f_1)+D_1$ and $\div(f_2)+D_2$
have no common components.
Then one has a surjection
\begin{eqnarray*}
H^0(D-D_1)
\ \oplus \ 
H^0(D-D_2)
&  \to &
H^0(D),
\\
(g_1,g_2) 
& \mapsto & 
g_1f_1 + g_2 f_2.
\end{eqnarray*}
\end{lemma}

\begin{proof}
First note that due to the assumptions, 
the assignments 
$\imath \colon h \mapsto (hf_2,-hf_1)$
and $\varphi \colon (h_1,h_2) \mapsto h_1f_1+h_2f_2$
give rise to an exact sequence of sheaves
$$
\xymatrix{
0\ar[r] 
& 
{\mathcal{O}_X(-D_1-D_2)}
\ar[r]^{\imath \qquad}
& 
{\mathcal{O}_X(-D_1)}
\oplus
{\mathcal{O}_X(-D_2)}
\ar[r]^{\qquad \qquad \varphi}
& 
{\mathcal{O}_X}
\ar[r] 
& 0.
}
$$
Tensoring this sequence with $\mathcal{O}_X(D)$ 
and looking at the associated cohomology sequence,
we obtain the assertion.
\end{proof}

\begin{proposition}
\label{gen1}
Let $X$ be a K3-surface, 
$w \in \Cl(X)$ be the class 
of a smooth irreducible 
curve $D \subseteq X$ of genus $g$
and consider the 
Veronese algebra
\begin{eqnarray*}
\mathcal{R}(X,w)
& := & 
\bigoplus_{n \in \NN} \mathcal{R}(X)_{nw}.
\end{eqnarray*}
Then the algebra $\mathcal{R}(X,w)$ 
is generated
in degree one if $D$ is not hyperelliptic 
or $g \le 1$,
in degrees one and three
if $g = 2$
and in degrees one and two
if $D$ is hyperelliptic of genus $g \ge 3$.
\end{proposition}

\begin{proof}
If $g=0$ holds, 
then $\mathcal{R}(X,w)=\CC[s]$ with $s\in H^0(D)$,
since $D$ is irreducible with negative self-intersection.
Thus  $\mathcal{R}(X,w)$ is generated in degree one.

For non-rational $D$, the canonical algebra 
$\oplus H^0(D, nK_D)$ is generated 
in degree one if $D$ is not hyperelliptic 
or $g = 1$,
in degrees one and three
if $g = 2$
and in degrees one and two
if $D$ is hyperelliptic and $g \ge 3$,
see~\cite[page~117]{ACGH}.
By the adjunction formula we have 
$\mathcal{O}_X(D)_{\vert D} \cong K_D$.
Thus, we obtain the exact sequence
$$
\xymatrix{
0
\ar[r]
&
H^0(X,(n-1)D)
\ar[r]
&
H^0(X,nD)
\ar[r]
&
H^0(D,nK_D)
\ar[r]
&
0,
}
$$
where the last zero is due to Kawamata-Viehweg 
vanishing theorem.
This gives the the assertion.
\end{proof}

In order to prove Theorem~\ref{gen-2}~(iii), 
we use the techniques introduced in~\cite{lv}.
We say that a degree $w \in R$ is 
{\em not essential\/}
for a minimal ideal $I$ of relations
of a $K$-graded algebra $R$ if no
minimal system of homogeneous generators 
of $I$ has members of degree $w$.

\begin{theorem}
\label{rels}
See~\cite{lv}.
Let $f_1, \ldots, f_r \in \mathcal{R}(X)$ 
be a minimal system of generators 
for the Cox ring of a surface $X$ 
and set $w_i := \deg(f_i) \in \Cl(X)$.
Consider the maps
\begin{eqnarray*}
\label{two}
\varphi_{w,i} \colon
\mathcal{R}(X)_{w-w_1-w_i}
\ \oplus \ 
\mathcal{R}(X)_{w-w_2-w_i}
&  \to &
\mathcal{R}(X)_{w-w_i},
\\
(g_1,g_2) 
& \mapsto & 
g_1f_1 + g_2 f_2,
\end{eqnarray*}
where $w \in \Cl(X)$ and $i=3, \ldots, r$.
If $w_1 \mal w_2 = 0$ holds and 
$\varphi_{w,i}$ is surjective for $i = 3, \ldots, r$,
then $w$ is not essential for 
the ideal of relations arising 
from~$f_1, \ldots, f_r$.
\end{theorem}

\begin{proof}[Proof of Theorem~\ref{gen-2}]
Let $D_i \in \WDiv(X)$ represent 
$w_i \in \Cl(X)$.
Then $D_i^2=0$ implies that the complete 
linear system of $D_i$ defines a fibration,
which in turn gives $h^0(w_i) = 2$.
In particular, we have bases
$(f_{i1}, f_{i2})$ for $\mathcal{R}(X)_{w_i}$, 
where $i = 1,2$.
Moreover, applying Riemann-Roch
yields $h^1(w_i) = 0$.

By Proposition~\ref{genpic2geneff}, 
the classes $w_1$ and $w_2$ generate 
the effective cone.
Moreover, $h^0(w_i) = 2$ and $w_i^2 = 0$
show that $w_i$ is semiample, i.e.,
we have $\Eff(X) = \SAmple(X)$.
Consequently, any divisor class
$w = aw_1 + bw_2$ with $a,b >0$
is ample and the 
Kawamata-Viehweg vanishing theorem 
gives $h^1(w) = 0$.

We show now that $\mathcal{R}(X)$ 
is generated in 
degrees $w_1$,$w_2$ and $w_1+w_2$.
Consider a class $w = aw_1 + bw_2$.
If $a \ge 3$ and $b \ge 1$ or 
$(a,b) = (2,1)$ holds, 
then we have $h^1(w-2w_1) = 0$. 
Thus, Lemma~\ref{gens}
provides a surjective map
$$
\varphi \colon
\mathcal{R}(X)_{w-w_1}
\oplus 
\mathcal{R}(X)_{w-w_1}
\ \to \ 
\mathcal{R}(X)_{w},
\qquad
(g_1,g_2) 
\ \mapsto \ 
g_1f_{11} + g_2f_{12}.
$$ 
If $b=0$ holds, then the complete linear system 
of any representative of $w$ is composed with 
a pencil. This implies again surjectivity
of the above map $\varphi$.

Iterating this procedure, we see that for any 
$w = aw_1 + bw_2$ with $a \ge 3$ and $b \ge 1$ or 
$(a,b) = (2,1)$ or $b=0$, 
the elements of $\mathcal{R}(X)_{w}$ 
are polynomials in $f_{11}, f_{12}$ and 
the elements of $\mathcal{R}(X)_{u}$,
where 
$$
u \ = \ 2w_1+bw_2
\quad \text{if } b\geq 2,
\qquad\qquad 
u \ = \ w_1+bw_2
\quad \text{if } b=0,1.
$$
Interchanging the roles of $a$ and $b$, in this
reasoning, we finally see that any element of
$\mathcal{R}(X)_{w}$ is a polynomial in 
$f_{11}, f_{12}, f_{21}, f_{22}$ and 
elements of $\mathcal{R}(X)_{nu}$,
where $u := w_1+w_2$ and $n \le 2$.

Thus, we are left with describing the elements
of $\mathcal{R}(X)_{nu}$, where $u = w_1+w_2$.
Observe that no complete linear system 
on $X$ has fixed components, because, 
by Lemma~\ref{negative} and the adjunction 
formula, any such component would
be a $(-2)$-curve and, by our assumptions, 
there are no classes of self-intersection $-2$
in $\Cl(X)$.
Moreover, note that $w_1$ and $w_2$ are the only 
classes of elliptic curves in $\Cl(X)$.

It follows that $u=w_1+w_2$ is represented by
a smooth irreducible curve $D \subseteq X$ of 
genus $u^2/2 + 1 > 3$.
Since $u^2 \geq 6$ holds, $u$ is a primitive 
class in $\Cl(X)$ and we have $u \mal w_i \geq 3$, 
the curve $D$ is not hyperelliptic, 
see~\cite[Thm.~5.2]{SD}.
According to Proposition~\ref{gen1} 
the elements of $\mathcal{R}(X)_{nu}$
are polynomials in those of
$\mathcal{R}(X)_{u}$.

Thus, we obtained that $\mathcal{R}(X)$ is 
generated in the degrees $w_1$, $w_2$ 
and $u := w_1+w_2$. Moreover, the 
Riemann-Roch Theorem gives us  
$$ 
\dim(\mathcal{R}(X)_{w_i})
\ = \ 
2,
\qquad
\dim(\mathcal{R}(X)_{w_1+w_2})
\ = \
k + 2.
$$

We turn to the relations.
First note that any minimal 
system of generators must
comprise a basis $(f_{11},f_{12})$
of $\mathcal{R}(X)_{w_1}$ 
and a basis $(f_{21},f_{22})$
of $\mathcal{R}(X)_{w_2}$.
Now, consider any degree 
$w = aw_1+bw_2$.
If $a\geq 4$ and $b\geq 2$ or $(a,b)=(3,2)$ 
holds, then we have
$$
h^1(w - 2w_1 - w_2) 
\ = \
h^1(w - 3w_1 - w_2) 
\ = \ 
0.
$$
Thus, taking $f_1 = f_{11}$ 
and $f_2 = f_{12}$ in Theorem~\ref{rels} 
and using 
Lemma~\ref{gens}, we see that 
$w$ is not essential for any minimal ideal 
of relations of $\mathcal{R}(X)$. 
If $b=1$ holds, then 
$$
\mathcal{R}(X)_{(c-1)w_1} 
\oplus 
\mathcal{R}(X)_{(c-1)w_1}
\ \to \
\mathcal{R}(X)_{cw_1},
\qquad
(g_1,g_2)
\ \mapsto \
g_1 f_{11} + g_2 f_{12}
$$
is surjective for $c = a, \, a-1$, because 
$\mathcal{R}(X)$ is generated in degrees 
$w_1$, $w_2$ and $w_1+w_2$.
Thus, Theorem~\ref{rels} shows $w$ 
is not essential for $b=1$. 
Eventually, there are no
relations of degree $w$ for $b=0$.
In fact, then $\mathcal{R}(X)_w$  
is generated by $f_{11}$ and $f_{12}$
and hence any such relation defines a 
relation among $f_{11}$ and $f_{12}$,
which contradicts the fact that 
$f_{11},f_{12}$ define a surjection
$X \to \PP_1$.

Exchanging the roles of $w_1$ and $w_2$ 
in this consideration, we obtain that 
essential relations can only occur in
degrees $2u$ and $3u$, where $u = w_1+w_2$.

In the case $k=3$, the statements proven
so far give that any minimal system of 
generators has five members and their degrees 
are $w_1$, $w_1$, $w_2$, $w_2$ and 
$u$. Hence there must be exactly one 
relation in $\mathcal{R}(X)$.
The degree of this relation minus
the sum of the degrees of the generators 
gives the canonical class, see~\cite[Prop.~8.5]{BeHa2},
and hence vanishes.
Thus our relation must have degree~$3u$. 

Finally, let $k \ge 4$.
As observed before, $\mathcal{R}(X)_{nu}$
is generated by $\mathcal{R}(X)_{u}$.
Hence, any relation in degree $nu$ is 
also a relation of 
$$ 
\bigoplus_{n \in \NN} \mathcal{R}(X)_{nu}.
$$
Since $u \cdot w_i >3$ holds and $X$ does
not contain smooth rational curves,
\cite[Thm.~7.2]{SD} tells us that the ideal 
of relations of this algebra is generated 
in degree $2$.
Thus, there are only essential relations 
of degree $2u$ in $\mathcal{R}(X)$.

In order to determine the dimension of 
$\mathcal{I}(X)_{2u}$ for a minimal 
ideal of relations $\mathcal{I}(X)$, 
note that we have the four generators
$f_{ij}$, where $1 \le i,j \le 2$, of
degree $w_i$, where $1 \le i \le 2$, 
and $k-2$ generators of degree 
$u = w_1+w_2$.
Using Riemann-Roch, we obtain
that $\mathcal{R}(X)_{2u}$ is of dimension 
$4k+2$.
Thus, denoting by $\CC[T]$
the polynomial ring in the above
generators
and by 
$V \subseteq \mathcal{R}(X)_u$ 
the vector space spanned by the $k-2$ 
generators we of degree $u$, 
we obtain
\begin{eqnarray*}
\dim(\mathcal{I}(X)_{2u}) 
& = &
\dim(\CC[T]_{2u})
\ - \ 
\dim(\mathcal{R}(X)_{2u})
\\
& = &
\dim({\rm Sym}^2 V) 
\ + \
4 (k-2)
\ + \
9 
\\
& = & 
\frac{k(k-3)}{2}.
\end{eqnarray*}
\end{proof}

We turn to the cases $w_1^2 = -2$ and 
$w_2^2 = 0,-2$. 
In contrast to the previous 
cases, the number of degrees occuring 
in a minimal system of generators for 
the Cox ring becomes arbitrary large
when $w_1 \mal w_2$ increases.

\begin{proposition}
\label{gens-0-2}
Let $X$ be a K3-surface with  
$\Cl(X) \cong \ZZ w_1 \oplus \ZZ w_2$
and intersection form given by
$w_1^2 = -2$, $w_2^2 = 0$ and
$w_1 \mal w_2 = k \in \NN$.
\begin{enumerate}
\item
The semiample cone $\SAmple(X)$ of $X$ 
is generated by the classes
$kw_1+2w_2$ and $w_2$.
\item
The Cox ring $\mathcal{R}(X)$ 
has generators in degrees 
$w_1$ and $aw_1 + w_2$, where 
$0 \le a \le \lfloor k/2\rfloor$.
\item
If $k > 1$ holds and $k$ is odd, 
then 
the Cox ring $\mathcal{R}(X)$ 
has, in addition to those of~(ii),
generators in degree $kw_1+2w_2$.
\end{enumerate}
\end{proposition} 

\begin{proof} 
To verify~(i), note that
$(aw_1+bw_2)\mal w_1 = -2a+kb$ 
and 
$(aw_1+bw_2)\mal w_2 = ka$ hold. 
These intersection products are both 
non-negative if $0\leq a\leq kb/2$ 
holds.
So the nef cone of $X$ is generated 
by the classes 
$kw_1+2w_2$ and $w_2$. 
Since $\SAmple(X)$ is polyhedral, 
the claim follows.

We prove~(ii).
By Proposition~\ref{genpic2geneff}, 
the classes $w_1$ and $w_2$ generate 
the effective cone.
Thus, $\mathcal{R}(X)$ has generators 
in the degrees $w_1$ and $w_2$. 
Note that $h^0(w_1) = 1$ and, 
fixing a generator 
$f_1 \in \mathcal{R}(X)_{w_1}$,
we obtain 
$\mathcal{R}(X)_{nw_1} = \CC f_1^n$.
Moreover, up to a constant, 
$f_1$ occurs in any minimal system
of generators of $\mathcal{R}(X)$;
we fix such a system 
$f_1, \ldots, f_r$.

We show now that $\mathcal{R}(X)$ 
has generators in degree 
$aw_1+w_2$ for any 
$1\leq a\leq [k/2]$.
By assertion~(i), 
the class $a w_1+ w_2$ 
is big and nef for 
$1\leq a\leq [k/2]$.
Using Riemann-Roch and the 
Kawamata-Viehweg vanishing theorem,
we obtain 
$$
h^0((a-1)w_1 + w_2)
\ < \ 
h^0(aw_1+w_2)
\quad
\text{for } 1 < a \le \lfloor k/2\rfloor.
$$
This implies that there exists an
$f \in \mathcal{R}(X)_{aw_1+w_2}$,
which is not a multiple of
$f_1 \in \mathcal{R}(X)_{w_1}$.
The same holds for $a=1$, because
then we have ($k=1$ implies $a=0$)
$$
h^0(w_2)
\ = \ 2,
\qquad
h^0(w_1+w_2)
\ = \ 
k+1
\ > \ 
2.
$$

Suppose that every monomial 
$m \in \mathcal{R}(X)_{aw_1+w_2}$ 
in the $f_i$ is a product $m=m_1m_2$ 
of non-constant $m_i$.
Then $m_1$ or $m_2$ belongs to 
$\mathcal{R}(X)_{bw_1} = \CC f_1^b$,
where $1\leq b\leq a$. 
Then every $f \in \mathcal{R}(X)_{aw_1+w_2}$ 
is a multiple of $f_1$, 
contradicting the previous statement.
Hence some $f_i$, where $i = 2, \ldots, r$ 
has degree $aw_1+w_2$.

We turn to~(iii). 
Then $kw_1+2w_2$ is big and nef.
Reasoning as before, we obtain  
that there is an $f \in \mathcal{R}(X)_{kw_1+2w_2}$
which is not a multiple of $f_1$. 
Suppose that every 
monomial $m\in\mathcal{R}(X)_{kw_1+2w_2}$
in the $f_i$ is a product $m=m_1m_2$ 
of non-constant $m_i$,
Then $m_1$ or $m_2$ belongs to 
$\mathcal{R}(X)_{bw_1+cw_2}$,
where $b/c \ge \lfloor k/2\rfloor+1$. 
Since $(bw_1+cw_2)\mal w_1 = -2b+kc < 0$
holds, every element of 
$\mathcal{R}(X)_{bw_1+cw_2}$
is divisible by $f_1$; 
a contradiction.
Hence some $f_i$, where $i = 2, \ldots, r$ 
has degree $kw_1+2w_2$.
\end{proof}

\begin{proposition}
\label{gens-2-2}
Let $X$ be a K3-surface with  
$\Cl(X) \cong \ZZ w_1 \oplus \ZZ w_2$
and intersection form given by
$w_1^2 = w_2^2 = -2$ and
$w_1 \mal w_2 = k \in \NN$.
\begin{enumerate}
\item
We have $k \ge 3$ and
the semiample cone $\SAmple(X)$ 
is generated by the classes
$kw_1+2w_2$ and $2w_1+kw_2$.
\item
The Cox ring $\mathcal{R}(X)$ 
has generators in degrees 
$aw_1+w_2$ and $w_1 + aw_2$, where 
$0 \le a \le \lfloor k/2\rfloor$.
\item
If $k > 1$ holds and $k$ is odd, 
then 
the Cox ring $\mathcal{R}(X)$ 
has, in addition to those of~(ii),
generators in degree $kw_1+2w_2$
and $2w_1+kw_2$.
\end{enumerate}
\end{proposition} 

\begin{proof}
By the Hodge Index Theorem,
$\Cl(X)$ has signature $(1,1)$, 
which implies $k\geq 3$.
Determining the semiample cone 
runs as in the proof of 
Proposition~\ref{gens-0-2}.

As to the remaining statements,
note that the semigroup 
$\SAmple(X)\cap\ZZ^2$ 
is generated by 
$aw_1+w_2$ and $w_1+aw_2$ 
with $0 \le a\le \lfloor k/2\rfloor$ 
if $k$ is even,
and by the same classes plus 
the two the extremal 
rays if $k$ is odd.
Reasoning as in the proof of 
Proposition~\ref{gens-0-2}
we obtain
\begin{eqnarray*}
h^0((a-1)w_1+w_2) 
& < & 
h^0(aw_1+w_2)
\end{eqnarray*}
whenever $0\leq a\leq \lfloor k/2\rfloor$
holds. 
This formula also holds, when 
$w_1$ and $w_2$ are exchanged. 
Now the same arguments as used in 
the proof of Proposition~\ref{gens-0-2} 
give the assertion.
\end{proof}

\begin{example}
Let $X$ be a K3-surface with  
$\Cl(X) \cong \ZZ w_1 \oplus \ZZ w_2$
and intersection form given by
$w_1^2 = w_2^2 = -2$ and
$w_1 \mal w_2 = 3$.
Then the  Cox ring $\mathcal{R}(X)$ 
has generators 
$$
f_{1,0}, f_{0,1}, f_{1,1}, g_{1,1}, f_{2,3}, f_{3,2}
$$
in the corresponding degrees by 
Proposition~\ref{gens-2-2} and its proof.
A monomial basis of 
$\operatorname{Sym}^3\mathcal{R}(X)_{w_1+w_2}$, 
plus $f_{2,3}f_{0,1}$ and $f_{3,2}f_{1,0}$, 
give $12$ linearly dependent elements of 
$\mathcal{R}(X)_{3w_1+3w_2}$ 
since this space has dimension $11$ by 
the Riemann Roch theorem. 
This means that $\mathcal{R}(X)$ 
has a relation in degree $3w_1+3w_2$.

Similarly, a monomial basis of 
$\operatorname{Sym}^5\mathcal{R}(X)_{w_1+w_2}$, 
plus $f_{2,3}f_{3,2}$ and the product of 
$f_{2,3}f_{1,0}$ for a monomial basis of 
$\operatorname{Sym}^2\mathcal{R}(X)_{w_1+w_2}$ 
give $28$ monomials. 
These are linearly dependent since the dimension 
of $\mathcal{R}(X)_{5w_1+5w_2}$ is $27$ by 
the Riemann Roch theorem. 
This means that $\mathcal{R}(X)$ has a 
relation in degree $5w_1+5w_2$.

We now give a geometric interpretation 
for generators and relations.
The map $\pi \colon X\to \PP^2$ associated 
to $w_1+w_2$ is a double cover branched 
along a smooth plane sextic, see~\cite{SD}.
Observe that $f_{1,0}f_{0,1}=\pi^*(s)$ and  
$f_{2,3}f_{3,2}=\pi^*(t)$, where $s=0$ is a 
line and $t=0$ a quintic in $\PP^2$.
The second equality gives a relation 
in degree $5w_1+5w_2$.
\end{example}

The assumptions $w_i^2 \in \{0,-2\}$ 
made in Theorem~\ref{gen-2}
imply that the primitive generators of the 
effective cone form a basis of the divisor 
class group.
However the techniques of its proof allow 
as well to treat, for example, the following
case, where the primitive generators of the 
effective cone span a sublattice of index
two in the divisor class group.

\begin{proposition}
\label{pic-eff}
Let $X$ be a K3-surface with  
$\Cl(X) \cong \ZZ w_1 \oplus \ZZ w_2$
and intersection form given by
$w_1^2 = 4$, $w_2^2 = -4$ and 
$w_1 \mal w_2 = 0$.
\begin{enumerate}
\item
The effective cone of $X$ is generated by 
$u_1 := w_1 + w_2$ and $u_2 := w_1 - w_2$.
\item
The Cox ring $\mathcal{R}(X)$ is generated
in degrees $u_1$, $u_2$ and $w_1$.
\item 
Any minimal ideal of relations 
of $\mathcal{R}(X)$ 
is generated in degree $2w_1$.
\end{enumerate}
\end{proposition}

\begin{proof}
Note that we have $u_1^2=u_2^2=0$.
Thus, by Riemann-Roch, we can assume 
that $u_1$ and $u_2$ are effective.
Moreover, Proposition~\ref{ne}
tells us that $\Eff(X)$ is polyhedral; 
we denote by $v_1$ and $v_2$ its primitive 
generators. Then we have $v_i^2 \in 4 \ZZ$
and thus Proposition~\ref{ko} gives
$v_1^2=v_2^2=0$. Choosing presentations
$u_i = a_iv_1 + b_iv_2$ with nonnegative 
$a_i, b_i \in \QQ$, we obtain
$$
8 
\ = \ 
u_1 \mal u_2 
\ = \ 
(a_1v_1 + b_1v_2) \mal (a_2v_1+b_2v_2)
\ = \ 
(a_1b_2 + a_2b_1) \mal v_1 \mal v_2,
$$
$$
0 
\ = \ 
u_i^2 
\ = \ 
(a_iv_1 + b_iv_2)^2
\ = \ 
2a_ib_i \, v_1 \mal v_2.
$$
The first identity gives $v_1 \mal v_2 \ne 0$ 
and, thus, the second one shows $a_ib_i = 0$.
As a consequence, we obtain 
$\{u_1,u_2\}=\{v_1,v_2\}$.
This proves the first assertion.

As to the second one, note that any effective 
divisor class $w \in \Cl(X)$ can be written as
$$
w
\ = \ 
 au_1 + bu_2 + cw_1,
\quad
\text{where }
a,b \in \NN, \ c=0,1.
$$
Observe that $u_1$ and $u_2$ are classes of 
elliptic curves. Moreover, we have $h^0(u_i)=2$
and thus $\mathcal{R}(X)_{u_i}$ has a basis
of the form  $(f_{i1},f_{i2})$.

We now proceed as in the proof of 
Theorem~\ref{gen-2}. 
If $a \ge 3$ and $b \ge 1$ hold, 
then $w - 2u_1$ is nef and big, 
and thus we have $h^1(w - 2u_1)=0$. 
Since $u_1^2=0$ holds, Lemma~\ref{gens} 
shows that the sections of $w$ are 
polynomials in $f_{11}$, $f_{12}$ 
and elements of $\mathcal{R}(X)_{w-u_1}$.

Iterating this procedure and interchanging 
the roles of $a$ and $b$, 
we reduce to study classes $w$ with $a+b\leq 4$. 
A case by case analysis now shows that 
$\mathcal{R}(X)$ is generated in degrees 
$u_1$, $u_2$ and $w_1$. 
In the following table we briefly provide 
the reason why $h^1(w-2u_1)=0$ holds,
when $a \ge b$ and $w-2u_1$ is not nef and big.
\begin{center}
\begin{tabular}{c|c|c|l}
$a$ & $b$ & $c$ & $h^1(w-2w_1)=0$ because\\[3pt]
\midrule
$1$ & $1$ & $0$ & $\mathcal{R}(X,w_1)$ is $1$-generated \\[3pt]
$1$ & $0$ & $1$ & $h^1(-u_2)=0$ \\[3pt]
$1$ & $1$ & $1$ & $\mathcal{R}(X,w_1)$ is $1$-generated \\[3pt]
$2$ & $0$ & $0$ & $\mathcal{R}(X,u_1)$ is $1$-generated \\[3pt]
$2$ & $1$ & $0$ & $h^1(u_2)=0$ \\[3pt]
\end{tabular}
\end{center}

In a similar way, Theorem~\ref{rels} and Lemma~\ref{gens} 
imply that $ w$ is not essential unless $a=b=1$.
Since $w_1$ is the class of a smooth irreducible 
curve of genus at most one, Proposition~\ref{gen1}
yields that $\mathcal{R}(X)_{w_1}$ is generated in degree one. 
This implies that the monomials $f_{1i}f_{2k}$ are 
quadratic functions in the $f_i \in \mathcal{R}(X)_{w_1}$:
\begin{eqnarray*}
f_{1i}f_{2k}
& = & 
q_{ik}(f_0,f_1,f_2,f_3),
\end{eqnarray*}
where $q_{ik}$ is a homogeneous polynomial of degree two. 
This gives $4$ independent relations in degree $(2,0)$ 
as can be checked since $h^0(2w_1)=10$ and the number 
of monomials of type $f_if_j$ and $f_{1i}f_{2k}$ is $14$.
\end{proof}

\section{Cox rings and coverings}
\label{sec:abcov}

In this section, we investigate the 
effect of certain, e.g.~cyclic,
coverings $\pi \colon X \to Y$ 
on the Cox ring.
Among other things, we obtain 
that finite generation of the Cox 
ring is preserved, provided that 
$\pi^*(\Cl(Y))$ is of finite index in 
$\Cl(X)$,
see Proposition~\ref{prop:abcovfingen}.
In the whole section, we work 
over an algebraically closed 
field $\KK$ of characteristic zero.
First, we make precise, which type of 
coverings we will treat.

\begin{construction}
\label{constr:abcov}
Let $Y$ be a normal variety and
$D_1, \ldots, D_r \in \CaDiv(Y)$ be 
a linearly independent collection of 
Cartier divisors.
Denote by $M^+ \subseteq \CaDiv(Y)$ the 
semigroup generated by the divisors
$D_1, \ldots, D_r$
and set
$$ 
Y(D_1, \ldots, D_r) 
\ := \
\Spec_Y (\mathcal{A}),
\qquad \qquad
\mathcal{A}
\ := \ 
\bigoplus_{D \in M^+}  \mathcal{O}_Y(-D).
$$
Then the inclusion $\mathcal{O}_Y \to \mathcal{A}$
defines a morphism $\alpha \colon Y(D_1, \ldots, D_r) \to Y$,
which is a (split) vector bundle of rank $r$ over $Y$.
Similarly, with $n_1, \ldots, n_r \in \ZZ_{>0}$
and $E_i := n_iD_i$, 
denote by $N^+ \subseteq \CaDiv(Y)$ the 
semigroup generated by $E_1, \ldots, E_r$.
Setting
$$ 
Y(E_1, \ldots, E_r) 
\ := \
\Spec_Y (\mathcal{B}),
\qquad \qquad
\mathcal{B}
\ := \ 
\bigoplus_{E \in N^+}  \mathcal{O}_Y(-E)
$$
gives a further (split) vector bundle 
$\beta \colon Y(E_1, \ldots, E_r) \to Y$ 
of rank $r$ over $Y$.
The inclusion $\mathcal{B} \subseteq \mathcal{A}$ 
defines a morphism 
$\kappa \colon Y(D_1, \ldots, D_r) \to Y(E_1, \ldots, E_r)$.
Now, let $\sigma \colon Y \to Y(E_1, \ldots, E_r)$
be a section such that all projections of 
$\sigma$ to the factors $Y(E_i)$ are nontrivial.
Then we define 
$$ 
X
\ := \ 
\kappa^{-1}(\sigma(Y))
\ \subseteq \ 
Y(D_1, \ldots, D_r). 
$$
Restricting $\alpha$ gives a morphism $\pi \colon X \to Y$;
which we call an {\em abelian covering\/} of $Y$.
Note that $\pi \colon X \to Y$ is the quotient for the 
action of the abelian group 
$\ZZ/n_1\ZZ \oplus \ldots \oplus \ZZ/n_r\ZZ$ on 
$X$ defined by the inclusion $\mathcal{B} \subseteq \mathcal{A}$ 
of graded algebras.
\end{construction}

\begin{remark}
For a smooth variety $Y$, 
Cartier divisors $D$ 
and $E := nD$ on $Y$
and a section 
$\sigma \colon Y \to Y(E)$ 
with non-trivial 
reduced divisor $B$, 
Construction~\ref{constr:abcov} 
gives a branched 
$n$-cyclic covering
of~Y, see~\cite[Sec.~1.17]{BPV}.
\end{remark}

In order to formulate our first result,
we need the following pullback construction 
for Weil divisors under an abelian 
covering $\pi \colon X \to Y$ 
of normal varieties.
Given $D \in \WDiv(Y)$, consider the 
restriction $D'$ of $D$ to the set 
$Y' \subseteq Y$ of smooth points and 
take the usual pullback $\pi^*(D')$
on $\pi^{-1}(Y')$. Since $\pi$ is finite,
the complement $X \setminus \pi^{-1}(Y')$
is of codimension at least two 
in $X$ and hence $\pi^*(D')$
uniquely extends to a Weil divisor 
$\pi^*(D)$ on $X$.

\begin{proposition}
\label{prop:abcov}
Let $\pi \colon X \to Y$ be 
an abelian covering as 
in~\ref{constr:abcov},
assume that $X$ is normal,
and let $K \subseteq \WDiv(Y)$ be
a subgroup containing 
$D_1, \ldots, D_r$ 
of~\ref{constr:abcov}.
Set
$$
\mathcal{S}_Y
\ := \ 
\bigoplus_{D \in K}
\mathcal{O}_Y(D),
\qquad\qquad
\mathcal{S}_X
\ := \ 
\bigoplus_{D \in K}
\mathcal{O}_X(\pi^*D).
$$
Then setting $\deg(T_i) := D_i$ turns 
$\mathcal{S}_Y[T_1, \ldots, T_r]$
into a $K$-graded sheaf of 
$\mathcal{O}_Y$-algebras and  
there is a $K$-graded isomorphism
of sheaves
\begin{eqnarray*}
\pi_* \mathcal{S}_X
& \cong &
\mathcal{S}_Y[T_1, \ldots, T_r]
/
\bangle{T_1^{n_1}-g_1, \ldots, T_r^{n_r}-g_r},
\end{eqnarray*}
where $g_i \in \Gamma(Y,\mathcal{O}(E_i))$ are 
sections
such that the branch divisor $B$ of the covering 
$\pi \colon X \to Y$ is given as 
\begin{eqnarray*}
B 
& = & 
\div(g_1) + \ldots + \div(g_r).
\end{eqnarray*}
\end{proposition}

\begin{proof}
Note that for any open set $V \subseteq Y$
and its intersection $V' := V \cap Y'$ with 
the set $Y' \subseteq Y$ of smooth points, 
the sections of $\mathcal{S}_Y$ over $V$
and $V'$ coincide
and also the sections of $\pi_* \mathcal{S}_X$ 
over $V$ and $V'$ coincide.
Hence, we may assume that $K \subseteq \WDiv(Y)$
consists of Cartier divisors.

A first step is to express the direct image 
$\pi_* \mathcal{S}_X$ 
in terms of $\pi_* \mathcal{O}_X$ and data living 
on $Y$.
Using the projection formula, we obtain
\begin{equation}
\label{eqn:cyclic1}
\pi_* \mathcal{S}_X 
\ = \
\pi_* 
\bigoplus_{D \in K} \mathcal{O}_X(\pi^*D)
\ \cong \ 
\bigoplus_{D \in K} 
\mathcal{O}_Y(D) \otimes_{\mathcal{O}_Y} \pi_* \mathcal{O}_X
\ \cong \
\mathcal{S}_Y \otimes_{\mathcal{O}_Y} \pi_* \mathcal{O}_X.
\end{equation}

Now we have to investigate 
$\pi_* \mathcal{O}_X$.
Denote by $q \colon \t{Y} \to Y$ 
the torsor associated to 
$\mathcal{S}_Y$, i.e., we have 
$\t{Y} = \Spec_Y(\mathcal{S}_Y)$.
Moreover in~\ref{constr:abcov}
we constructed the rank $r$ 
vector bundles
$$
\alpha \colon Y(D_1, \ldots, D_r) \ \to \ Y,
\qquad\qquad
\beta \colon Y(E_1, \ldots, E_r) \ \to \ Y.
$$
Using the pullback divisors
$q^*(D_i)$ and $q^*(E_i)$,
we obtain the respective pullback bundles
$$
\t{\alpha} \colon
\t{Y}(q^*D_1, \ldots, q^*D_r) \ \to \ \t{Y},
\qquad\qquad
\t{\beta} \colon\t{Y}(q^*E_1, \ldots, q^*E_r) \ \to \ \t{Y}.
$$
Set for short $Y({\bf D}) := Y(D_1, \ldots, D_r)$ 
and  $\t{Y}(q^*{\bf D}) := \t{Y}(q^*D_1, \ldots, q^*D_r)$.
Similarly, define $Y({\bf E})$ and $\t{Y}(q^*{\bf E})$.
Then we have a commutative diagram
$$ 
\xymatrix{
{\t{X}}
\ar[rr]^{\t{\imath}}
\ar[d]_{p}
& &
{\t{Y}}(q^*{\bf D})
\ar[rr]^{\t{\kappa}}
\ar@/^2pc/@<1ex>[rrrr]^{\t{\alpha}}
\ar[d]^{q_{\bf D}}
& &
{\t{Y}}(q^*{\bf E})
\ar[rr]_{\qquad \t{\beta}}
\ar@{<-}@<1ex>[rr]^{\qquad \t{\sigma}}
\ar[d]_{q_{\bf E}}
& &
{\t{Y}}
\ar[d]^{q}
\\
X
\ar[rr]_{\imath}
\ar@/_3pc/[rrrrrr]_{\pi}
& &
Y({\bf D})
\ar[rr]_{\kappa}
& &
Y({\bf E})
\ar[rr]_{\qquad \beta}
\ar@{<-}@<1ex>[rr]^{\qquad \sigma}
& &
Y
}
$$
where $p$, $q_{\bf D}$ and $q_{\bf E}$ are 
the canonical morphisms,
we set $\t{X} := q_{\bf D}^{-1}(X)$
and $\t{\sigma} := q^* \sigma$ 
is the pull back section.

Recall that $q \colon \t{Y} \to Y$ is 
the quotient for the free action of
the torus $H := \Spec(\KK[K])$ 
defined by the grading of 
$q_* \mathcal{O}_{\t{Y}} = \mathcal{S}_Y$.
Thus, $\t{Y}({\bf D})$ and $\t{X}$ 
inherit free $H$-actions having 
$q_{\bf D}$ and $p$ as their respective 
quotients.
Moreover, let 
$\t{\mathcal{I}}$
denote the ideal sheaf of $\t{X}$
in $\t{Y}(q^*{\bf D})$. 
Then $\t{\mathcal{I}}$ is homogeneous, and we have
\begin{equation}
\label{eqn:cyclic2}
\pi_* \mathcal{O}_X
\ \cong \
\left(
\pi_* p_*  \mathcal{O}_{\t{X}} 
\right)_0
\ \cong \
\left(
\pi_* p_* \t{\imath}^*
\left(\mathcal{O}_{\t{Y}(q^*{\bf D})} / \t{\mathcal{I}}\right)
\right)_0
\ = \
\left(
q_* \t{\alpha}_* 
\left(\mathcal{O}_{\t{Y}(q^*{\bf D})} / \t{\mathcal{I}}\right)
\right)_0.
\end{equation}

To proceed, we need a suitable 
trivialization of the bundle 
$\t{\alpha} \colon {\t{Y}}(q^*{\bf D}) \to \t{Y}$.
For this, consider an open affine subset 
$V \subseteq Y$ such that on $V$ 
we have 
$D_i = \div(h_{i,V})$ for $1 \le i \le r$.
This gives us sections
$$ 
\eta_{i,V} \ := \ h_{i,V}^{-1} 
\ \in \ 
\Gamma(V,\mathcal{O}_Y(D_i))
\ \subseteq \
\Gamma(q^{-1}(V),\mathcal{O}_{\t{Y}})_{D_i},
$$
$$
q_{\bf D}^* (h_{i,V})
\ \in \
q_{\bf D}^* 
\left(\Gamma(\alpha^{-1}(V),\mathcal{O}_{Y({\bf D})})\right)
\ \subseteq \
\Gamma(\t{\alpha}^{-1}(q^{-1}(V)),\mathcal{O}_{\t{Y}(q^*{\bf D})})_0 
$$
Given another open affine subset 
$W \subseteq Y$ such that on $W$ 
we have 
$D_i = \div(h_{i,W})$ for $1 \le i \le r$,
we obtain over $V \cap W$ for the corresponding
sections:
$$ 
\frac{\t{\alpha}^*(\eta_{i,W})}{\t{\alpha}^*(\eta_{i,V})}
\ = \ 
\t{\alpha}^* q^* 
\left( 
\frac{\eta_{i,W}}{\eta_{i,V}}
\right)
\ = \ 
q_{\bf D}^* \alpha^*
\left( 
\frac{h_{i,V}}{h_{i,W}}
\right)
\ = \ 
\frac{q_{\bf D}^* (h_{i,V})}{q_{\bf D}^* (h_{i,W})}.
$$
Covering $Y$ with  $V$'s
as above, we obtain that the functions 
$\t{\alpha}^*(\eta_{i,V}) \cdot q_{\bf D}^* (h_{i,V})$ 
living on $\t{\alpha}^{-1}(q^{-1}(V))$ glue 
together to a global regular function 
$f_i$ of degree $D_i$ on $\t{Y}(q^*D_i)$
generating $\mathcal{O}_{\t{Y}}(q^*D_i)$ 
over $\mathcal{O}_{\t{Y}}$. 
Thus, the $f_i$ define a trivialization
$$ 
\xymatrix{
&&
{\t{Y} \times \KK^r}
\ar[rr]^{(\t{y},z) \mapsto (\t{y},z^n)}
& &
{\t{Y} \times \KK^r}
\ar@<1ex>@{<-}@/^1pc/[drr]^{\qquad \id \times g}
& &
\\
{\t{X}}
\ar[rr]^{\t{\imath}}
& &
{\t{Y}}(q^*{\bf D})
\ar[rr]_{\t{\kappa}}
\ar[u]_{\cong}^{\t{\alpha} \times f}
& &
{\t{Y}}(q^*{\bf E})
\ar[rr]_{\qquad \t{\beta}}
\ar@{<-}@<1ex>[rr]^{\qquad \t{\sigma}}
\ar[u]_{\cong}^{\t{\beta} \times f^n}
& &
{\t{Y}}
\ar@{<-}@/_1pc/[ull]^{\pr_{\t{Y}}}
}
$$
where we write
$z$ for $(z_1,\ldots, z_r)$ and
$z^n$ for $(z_1^{n_1}, \ldots, z_r^{n_r})$
etc..
Since $\t{\sigma} = q^*\sigma$ is $H$-equivariant, 
each component $g_i$ of $g$ is homogeneous 
with $\deg(g_i) = E_i$.
Note that the divisors $\div(g_i)$  
describe the branch divisor as claimed.

Denote by 
$\t{\mathcal{J}} \subseteq \mathcal{O}_{\t{Y}}[T_1,\ldots,T_r]$ 
the ideal 
sheaf of the image of $\t{X}$ in $\t{Y} \times \KK^r$.
Then, using
$\t{X}  = \t{\kappa}^{-1}(\t{\sigma}(\t{Y}))$,
we obtain
$$
\t{\mathcal{I}} 
\ = \
\bangle{
f_1^{n_1}-\t{\alpha}^*g_1,\, \ldots, \, f_r^{n_r}-\t{\alpha}^*g_r
},
\qquad\qquad
\t{\mathcal{J}} 
\ = \
\bangle{
T_1^{n_1}- g_1,\, \ldots, \, T_r^{n_r}-g_r
}.
$$
Thus, using the isomorphism 
$q_*\t{\alpha}_* \mathcal{O}_{\t{Y}(q^*{\bf D})}
\cong 
\mathcal{S}_Y[T_1,\ldots,T_r]$
established by the above commutative diagram,
we may continue~(\ref{eqn:cyclic2}) as 
\begin{equation}
\label{eqn:cyclic3}
\pi_* \mathcal{O}_X
\ \cong \
\left(
\mathcal{S}_Y[T_1, \ldots, T_r] / \t{\mathcal{J}}
\right)_0
\ \cong \
\mathcal{S}_Y[T_1, \ldots, T_r]_0 / \t{\mathcal{J}}_0
\end{equation}
The homogeneous ideal sheaf $\t{\mathcal{I}}$ is 
locally, over $Y$, generated in degree zero in the sense
that we have 
$\t{\mathcal{I}} = 
\t{\alpha}^* \mathcal{O}_{\t{Y}} \cdot \t{\mathcal{I}}_0$.
The same holds for $\t{\mathcal{J}}$, and we obtain
$$ 
\pi_* \mathcal{S}_X 
\ \cong \
\mathcal{S}_Y 
\otimes_{\mathcal{O}_Y} 
\pi_* \mathcal{O}_X
\ \cong \ 
\mathcal{S}_Y 
\otimes_{\mathcal{O}_Y}
\mathcal{S}_Y[T_1, \ldots, T_r]_0 / \t{\mathcal{J}}_0
\ \cong \
\mathcal{S}_Y[T_1, \ldots, T_r] / \t{\mathcal{J}}.
$$
\end{proof}

\begin{proposition}
\label{prop:vafin2crfin}
Consider a normal variety $X$,
a finitely generated subgroup 
$K \subseteq \WDiv(X)$ mapping 
onto $\Cl(X)$, a subgroup $L \subseteq K$ 
and the algebras
$$ 
R \ := \ \bigoplus_{D \in K} \Gamma(X,\mathcal{O}_X(D)),
\qquad\qquad
A \ := \ \bigoplus_{D \in L} \Gamma(X,\mathcal{O}_X(D)).
$$
If the subgroup $L \subseteq K$ is of finite index
and the algebra $A$ is 
finitely generated, then also the algebra 
$R$ is finitely generated.
\end{proposition}

\begin{lemma}
\label{lem:vafin2crfin}
Let $X$ be a normal variety,
$K$ be a finitely generated abelian group,
$\mathcal{R}$ be a quasicoherent 
sheaf of normal 
$K$-graded $\mathcal{O}_X$-algebras
and $Z$ be its relative spectrum,
$$
\mathcal{R} 
\ = \ 
\bigoplus_{w \in K} \mathcal{R}_w,
\qquad
\qquad
Z 
\ := \ 
\Spec_{X}(\mathcal{R}),
$$
where we assume $\mathcal{R}$ to be 
locally of finite type.
Let $L \subseteq K$ be a subgroup of finite 
index and consider the associated Veronese 
subalgebra of the algebra of global sections
$$ 
A 
\ := \ 
\bigoplus_{w \in L} \Gamma(X,\mathcal{R}_w)
\ \subseteq \ 
\bigoplus_{w \in K} \Gamma(X,\mathcal{R}_w)
\ =: \ 
R.
$$
Suppose that there are homogeneous sections 
$f_1, \ldots, f_r \in A$ 
such that each $R_{f_i}$ is finitely
generated, 
each
$Z_{f_i} = Z \setminus V(Z,f_i)$ is an affine
variety and we have
\begin{eqnarray*}
Z 
& = & 
Z_{f_1} \cup \ldots \cup Z_{f_r}.
\end{eqnarray*}
If $A$ is finitely generated and 
for $Y := \Spec(A)$ 
the complement  
$Y \setminus (Y_{f_1} \cup \ldots \cup Y_{f_r})$
is of codimension at least two in $Y$,
then $R$ is finitely generated.
\end{lemma}

\begin{proof}
First note that $Z$ is a variety with 
$\Gamma(Z, \mathcal{O}) = R$ and that
$\Gamma(Z_{f_i}, \mathcal{O}) = R_{f_i}$ 
holds.
Since each $R_{f_i}$ is finitely 
generated, we may construct a finitely generated 
$K$-graded subalgebra $S \subseteq R$ with
$$ 
A \ \subseteq \ S,
\qquad \qquad
S_{f_i} \ = \ R_{f_i}
\quad \text{ for } 
1 \le i \le r.
$$

Set $Z' := \Spec(S)$.
Then the inclusion $S \subseteq R$ 
defines a canonical 
morphism $\imath \colon Z \to Z'$.
Moreover, $S$ is canonically graded by 
the factor group $K/L$ 
and hence $Z'$ comes with an action 
of the finite abelian group $G := \Spec(\KK[K/L])$.
The inclusion $A \subseteq S$ defines the 
quotient map $\pi \colon Z' \to Y$ 
for the action of $G$.
By construction, we have 
$$ 
Z_{f_i}
\ = \ 
\imath^{-1}(Z'_{f_i})
\ \cong \ 
Z'_{f_i},
\qquad\qquad
Z'_{f_i}
\ = \ 
\pi^{-1}(Y_{f_i}).
$$
Consequently, $\imath \colon Z \to Z'$ is an 
open embedding, and we may regard $Z$ 
as a subset of $Z'$.
By our assumptions, 
setting $V := Y_{f_1} \cup \ldots \cup Y_{f_r}$, 
we obtain $Z = \pi^{-1}(V)$.
Since $\pi \colon Z' \to Y$ is a finite map,
we can conclude
$$ 
\dim(Z' \setminus Z)
\ = \ 
\dim(Y \setminus V) 
\ \le \ 
\dim(Y) - 2
\ = \ 
\dim(Z') - 2.
$$
Let $Z'' \to Z'$ be the normalization.
Since $Z$ is normal, we have 
$Z \subseteq Z''$.
We conclude that 
$R = \Gamma(Z,\mathcal{O}) = \Gamma(Z'',\mathcal{O})$
holds, and thus $R$ is finitely 
generated.
\end{proof}

\begin{proof}[Proof of Proposition~\ref{prop:vafin2crfin}]
First note that the rings $R$ and $A$ do not change if 
we replace $X$ with the set of its smooth points.
Thus, we may assume that $X$ is smooth.
Then we obtain graded sheaves of normal
$\mathcal{O}_X$-algebras
$$
\mathcal{R}  
\ = \ 
\bigoplus_{w \in K} \mathcal{R}_w,
\qquad \qquad
\mathcal{A}  = \bigoplus_{w \in L} \mathcal{R}_w,
$$
which are locally of finite type.
Our task is to verify the assumptions 
of Lemma~\ref{lem:vafin2crfin} 
for $\mathcal{R}$ and $A = \Gamma(X,\mathcal{A})$.
Setting $Z  := \Spec_X(\mathcal{R})$ 
and
$\t{X} :=  \Spec_X(\mathcal{A})$,
we obtain normal varieties, 
and we have a commutative diagram 
of affine morphisms
$$ 
\xymatrix{
Z
\ar[rr]^{\kappa}
\ar[dr]_{p}
&&
{\t{X}}
\ar[dl]^{q}
\\
&
X 
&
}
$$
where $\kappa$ is the quotient for the action of 
the finite abelian group $G := \Spec(\KK[K/L])$
on $Z$ defined by the canonical $(K/L)$-grading 
of $\mathcal{R}$.
Moreover, we have an affine variety $Y := \Spec(A)$, 
and there is a canonical morphism 
$\imath \colon \t{X} \to Y$.

To obtain the desired sections $f_i \in A$, 
cover $X$ by affine open subsets 
$U_1, \ldots, U_r$.
Then each $X \setminus U_i$ is the set of zeroes 
of a suitable homogeneous section $f_i \in R_i$. 
Replacing $f_i$ with a suitable power, we achieve
$f_i \in A_i$. 
Thus, we can cover $\t{X}$ by the 
affine open subsets
$$
q^{-1}(U_i) 
\ = \ 
\t{X}_{f_i}
\ = \ 
\imath^{-1}(Y_{f_i});
$$ 
use e.g.~\cite[Lemma~2.3]{Ha2}. 
Now, 
$\Gamma(\t{X},\mathcal{O}) 
= A = 
\Gamma(Y, \mathcal{O})$ 
implies
$\t{X}_{f_i} \cong Y_{f_i}$.
Thus, each restriction 
$\imath \colon \t{X}_{f_i} \to Y_{f_i}$
is an isomorphism 
and $\imath \colon \t{X} \to Y$
is an open embedding.
Moreover, using again
$\Gamma(\t{X}, \mathcal{O}) 
=
\Gamma(Y, \mathcal{O})$,
we see that we have a small complement 
$$
Y \setminus \t{X}
\ = \
Y \setminus (\t{X}_{f_1} \cup \ldots \cup \t{X}_{f_r})
\ = \ 
Y \setminus (Y_{f_1} \cup \ldots \cup Y_{f_r}).
$$
\end{proof}

\begin{proposition}
\label{prop:abcovfingen}
Let $\pi \colon X \to Y$ be 
an abelian covering of 
normal varieties with finitely 
generated free divisor class groups
such that $\pi^*(\Cl(Y))$ 
is of finite index in $\Cl(X)$.
Then the following statements 
are equivalent.
\begin{enumerate}
\item
The Cox ring $\mathcal{R}(X)$ is 
a finitely generated $\KK$-algebra.
\item
The Cox ring $\mathcal{R}(Y)$ is 
a finitely generated $\KK$-algebra.
\end{enumerate}
\end{proposition}

\begin{proof}
Let $M \subseteq \WDiv(Y)$ and 
$K \subseteq \WDiv(X)$ be subgroups 
mapping isomorphically to the respective 
divisor class groups $\Cl(Y)$ and $\Cl(X)$.
Then the Cox rings are given as  
$$ 
\mathcal{R}(Y) 
\ = \ 
\bigoplus_{E \in M} \Gamma(Y,\mathcal{O}_Y(E)),
\qquad\qquad
\mathcal{R}(X) 
\ = \ 
\bigoplus_{D \in K} \Gamma(X,\mathcal{O}_X(D)).
$$

Since $\Cl(Y)$ is free and 
$\pi \colon X \to Y$ is the quotient for a finite
group action, the pullback $\pi^* \colon \Cl(Y) \to \Cl(X)$ 
is injective, see~\cite[Ex.~1.7.6]{Fu}.
Consequently, there are a unique subgroup 
$L \subseteq K$ 
and an isomorphism $\pi^*(M) \to L$ inducing the 
identity on $\pi^*(\Cl(Y))$.
By our assumption, $L$ is of finite index in $K$.
Moreover, we have canonical identifications
$$ 
\mathcal{R}(Y)
\ \subseteq \
\bigoplus_{E \in M} \Gamma(X,\mathcal{O}_X(\pi^*(E)))
\ = \
S := \bigoplus_{D \in L} \Gamma(X,\mathcal{O}_X(D))
\ \subseteq \
\mathcal{R}(X).
$$

Suppose that $\mathcal{R}(X)$ is finitely 
generated over $\KK$. 
Then also the Veronese subalgebra 
$S \subseteq \mathcal{R}(X)$ is 
finitely generated over $\KK$.
Moreover, by Proposition~\ref{prop:abcov},
the algebra $S$ is a finite module over 
$\mathcal{R}(Y)$. Thus, the tower 
$\KK \subseteq \mathcal{R}(Y) \subseteq S$ 
fullfills the assumptions of the Artin-Tate
Lemma~\cite[Prop.~7.8]{AM}, 
and we obtain that $\mathcal{R}(Y)$
is a finitely generated $\KK$-algebra.

Now let $\mathcal{R}(Y)$ be finitely 
generated over~$\KK$. 
Then  Proposition~\ref{prop:abcov} tells 
us that $S$ is finitely generated over~$\KK$.
Thus Proposition~\ref{prop:vafin2crfin} 
shows that $\mathcal{R}(X)$ 
is finitely generated over~$\KK$.
\end{proof}

\section{Cox rings and blowing up}
\label{sec:coxblow}

In this section, we compute the Cox ring 
of the fourth Hirzebruch 
surface blown up at three general points.
As in the preceding section, we work 
over an algebraically closed 
field $\KK$ of characteristic zero.
We use the technique of toric ambient 
modifications provided in~\cite{Ha2},
and begin with giving short outline of this
technique.
A basic ingredient is the following 
construction of the Cox ring and the 
universal torsor of 
a toric variety given in~\cite{Cox}.

\begin{construction}
\label{constr:cox}
Let $Z$ be the toric variety arising 
from a complete fan $\Sigma$ in a 
lattice~$N$,
and suppose that the primitive generators
$v_1, \ldots, v_r$ of $\Sigma$ 
span the lattice~$N$.
Then we have mutually dual
exact sequences
$$ 
\xymatrix{
0
\ar[rr]
&&
L 
\ar[rr]
&&
{\ZZ^r}
\ar[rr]^{P \colon e_i \mapsto v_i}
&&
N
\ar[rr]
&&
0,
\\
0
&&
K
\ar[ll]
&&
{\ZZ^r}
\ar[ll]^{Q}
&&
M
\ar[ll]
&&
0,
\ar[ll]
}
$$
where the lattice $K$ is isomorphic to the 
divisor class group $\Cl(Z)$.
The Cox ring of $Z$ is the polynomial ring 
$\mathcal{R}(Z) =\KK[T_1,\ldots,T_r]$ 
with the $K$-grading
defined by $\deg(T_i) := Q(e_i)$.
Moreover,  denoting by $\delta \subseteq \QQ^r$ 
the positive orthant, we obtain a fan in $\ZZ^r$ 
consisting of certain faces of $\delta$, namely
\begin{eqnarray*}
\rq{\Sigma}
& := & 
\left\{
\rq{\sigma} \preceq \delta; \; P(\rq{\sigma}) \subseteq \sigma 
\text{ for some } \sigma \in \Sigma 
\right\}.
\end{eqnarray*}
The associated toric variety $\rq{Z}$  
is an open toric 
subvariety of $\b{Z} := \KK^r$.
The toric morphism $p \colon \rq{Z} \to Z$ 
defined by $P \colon \ZZ^r \to N$ is 
a universal torsor; it is a quotient for 
the action of the torus $\Spec(\KK[K])$
on $\b{Z}$ defined by the $K$-grading
of $\mathcal{R}(Z)$.
\end{construction}

Given a variety $X_0$,
the rough idea of~\cite{Ha2}
is to work with a suitable 
embedding $X_0 \subseteq Z_0$ into 
a toric variety,
consider the proper transform
$X_1 \subseteq Z_1$ under suitable
toric modifications $Z_1 \to Z_0$
and then compute
the Cox ring $\mathcal{R}(X_1)$ 
in terms of $\mathcal{R}(X_0)$ 
using the toric universal
torsors over $Z_1$ and $Z_0$.
In our outline, we restrict 
to the case of blowing up a
smooth projective surface $X_0$ 
with divisor class group $K_0 \cong \ZZ^{k_0}$ 
and a Cox ring, which admits a 
representation 
$$ 
\mathcal{R}(X_0) 
\ = \ 
\KK[T_1, \ldots, T_r] / \bangle{f_0},
\qquad
\deg(T_i) \ = \ w_i \ \in \ K_0,
$$
where $f_0 \in \KK[T_1, \ldots, T_r]$ is a 
homogeneous polynomial and the $T_i$ define
pairwise nonassociated prime elements
in $\mathcal{R}(X_0)$.
We use this presentation to embed $X_0$ 
into a toric variety. 
First note that we have mutually dual
sequences
$$ 
\xymatrix{
0
\ar[rr]
&&
M 
\ar[rr]
&&
{\ZZ^r}
\ar[rr]^{Q_0 \colon e_i \mapsto w_i}
&&
K_0
\ar[rr]
&&
0,
\\
0
&&
N 
\ar[ll]
&&
{\ZZ^r}
\ar[ll]^{P_0}
&&
L_0
\ar[ll]
&&
0.
\ar[ll]
}
$$
Consider any complete simplicial 
fan $\Sigma_0$ in $N$
having the images $v_i := P_0(e_i) \in N$ of 
the canonical base vectors $e_i \in \ZZ^r$ 
as the generators of its rays.
Let $\rq{\Sigma}_0$ be the fan 
consisting of faces of the positive orthant 
$\delta \subseteq \QQ^r$ provided by 
Construction~\ref{constr:cox}.
Then the toric variety $\rq{Z}_0$ is an open toric 
subvariety of $\b{Z}_0 := \KK^r$, and the 
toric morphism $p_0 \colon \rq{Z}_0 \to Z_0$ 
defined by $P_0 \colon \ZZ^r \to N$ is 
a universal torsor. 
Moreover, setting 
$$
\b{X}_0 \  := \ V(\b{Z}_0,f_0),
\qquad\qquad
\rq{X}_0 \ := \ \b{X}_0 \cap \rq{Z}_0,
$$
we obtain $X_0 \cong p_0(\rq{X}_0)$, which allows us
to view $X_0$ as a closed subvariety of the toric
variety $Z_0$, see~\cite[Prop.~3.14]{Ha2}.
Note that our freedom in choosing the fan 
$\Sigma_0$ essentially relies on the assumption
that $X_0$ is a surface; in general one 
has to proceed more carefully, as 
$X_0$ and $p_0(\rq{X}_0)$ may differ by 
a small birational transformation.
 
Now we perform a toric modification.
Suppose that for some $d \ge 2$
the cone $\sigma_0$ spanned
by $v_1, \ldots, v_d$
belongs to $\Sigma_0$ and that
we have a primitive lattice vector 
$$
v_\infty 
\ := \
v_1 + \ldots + v_d
\ \in \
N.
$$ 
Recall that $\sigma_0$ corresponds to a toric orbit 
$T_0 \mal z_0 \subseteq Z_0$.
Moreover the stellar subdivision
$\Sigma_1$ of $\Sigma_0$ at $v_\infty$
defines a modification
$Z_1 \to Z_0$ of toric varieties having 
the closure of the toric orbit 
$T_0 \mal z_0 \subseteq Z_0$ as its center.
Then we have commutative diagrams
$$ 
\xymatrix{
{\b{Z}_1}
\ar[r]^{\b{\pi}}
&
{\b{Z}_0}
\\
{\rq{Z}_1}
\ar[r]
\ar[u]
\ar[d]_{p_1}
&
{\rq{Z}_0}
\ar[u]
\ar[d]^{p_0}
\\
Z_1
\ar[r]_{\pi}
&
Z_0
}
\qquad
\qquad
\qquad
\xymatrix{
{\b{X}_1}
\ar[r]
&
{\b{X}_0}
\\
{\rq{X}_1}
\ar[r]
\ar[u]
\ar[d]
&
{\rq{X}_0}
\ar[u]
\ar[d]
\\
X_1
\ar[r]
&
X_0
}
$$
where $p_1 \colon \rq{Z}_1 \to Z_1$ denotes 
the toric universal torsor,
$X_1 \subseteq Z_1$ the proper transform of 
$X_0 \subseteq Z_0$
and we write
$\rq{X}_1 = p_1^{-1}(X_1)$ for the inverse image and 
$\b{X}_1$ for the closure of
$\rq{X}_1$ in $\b{Z}_1 = \KK^{r+1}$.
Note that we have
\begin{eqnarray*} 
\b{\pi}(z_1, \ldots, z_{\infty})
& = & 
(z_1z_\infty, \ldots, z_dz_\infty, z_{d+1}, \ldots, z_r)
\end{eqnarray*}
for the lifting $\b{\pi} \colon \b{Z}_1 \to \b{Z}_0$ 
of the toric modification $\pi \colon Z_1 \to Z_0$ 
to the total coordinate spaces,
see~\cite[Lemma.~5.3]{Ha2}.
Moreover, $p_1 \colon \rq{Z}_1 \to Z_1$ 
defines another pair of dual sequences
$$ 
\xymatrix{
0
\ar[rr]
&&
L_1
\ar[rr]
&&
{\ZZ^{r+1}}
\ar[rr]^{P_1 \colon e_i \mapsto v_i}
&&
N
\ar[rr]
&&
0
\\
0
&&
K_1 
\ar[ll]
&&
{\ZZ^{r+1}}
\ar[ll]^{Q_1}
&&
M
\ar[ll]
&&
0
\ar[ll]
}
$$

Now, the basic observation is that under some 
mild assumptions, $\b{X}_1$ is the total 
coordinate space of $X_1$ and the 
explicit description of $\b{\pi}$ given above,
enables us to compute moreover the Cox ring.
For the precise statement,
consider the $\ZZ_{\ge 0}$-grading of
$\KK[T_1, \ldots, T_r]$ given by 
\begin{eqnarray*}
\deg(T_i) 
& := &
\begin{cases}
1 &   1 \le i  \le d,
\\
0 & d+1 \le i \le r.
\end{cases}
\end{eqnarray*}
Then we can write 
$f_0 = g_{k_0} + \ldots + g_{k_m}$ 
with $g_{k_i}$ homogeneous of degree
$k_i \in \ZZ_{\ge 0}$ and 
$k_0 < \ldots < k_m$.
We call
$f_0 \in \KK[T_1, \ldots, T_r]$
{\em admissible}, if 
$g_{k_0}$ is an irreducible 
polynomial in at least two variables 
and, moreover, 
$\b{X}_0 = V(f_0)$ intersects the toric orbit 
$0 \times \TT^{r-d}$ of $\KK^r$.
Then~\cite[Prop.~7.2]{Ha2} says the following.

\begin{proposition}
\label{recipe}
Suppose that the polynomial 
$f_0 \in \KK[T_1, \ldots, T_r]$
is admissible.
Then the proper transform $X_1$ has $\b{X}_1$ 
as its total coordinate space,
and the Cox ring of $X_1$ is given as
$$ 
\mathcal{R}(X_1) 
\ = \ 
\KK[T_1, \ldots, T_r,T_\infty] / \bangle{f_1},
\qquad
f_1 
\ := \ 
\frac{f_0(T_1T_\infty, \ldots, T_dT_\infty,T_{d+1}, \ldots, T_r)}{T_\infty^d},
$$
where $d \in \NN$ is maximal such that $f_1$ stays a polynomial.
The divisor class group of $X_1$ is given by 
$\Cl(X_1) \cong K_1$ and the degrees of the variables 
$T_1, \ldots, T_r,T_\infty$ is given by 
$\deg(T_i) = Q_1(e_i)$.
\end{proposition}

We turn to the 
fourth Hirzebruch surface $\FF_4$.
If $q \colon \FF_4 \to \PP_1$ denotes the
bundle projection, we write
$C_1, C_2, C_3, C_5  \subseteq  \FF_4$
for the section at infinity,
the fiber $q^{-1}(0)$,
 the fiber $q^{-1}(\infty)$
and the zero section
respectively.
 As a toric variety, $\FF_4$ arises from the fan
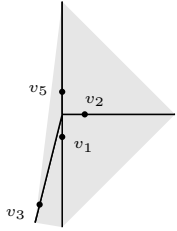
\begin{center}
\begin{picture}(0,0)%
\includegraphics{fbblow.pstex}%
\end{picture}%
\setlength{\unitlength}{1243sp}%
\begingroup\makeatletter\ifx\SetFigFontNFSS\undefined%
\gdef\SetFigFontNFSS#1#2#3#4#5{%
  \reset@font\fontsize{#1}{#2pt}%
  \fontfamily{#3}\fontseries{#4}\fontshape{#5}%
  \selectfont}%
\fi\endgroup%
\begin{picture}(3423,4566)(2911,-4594)
\put(4276,-2986){\makebox(0,0)[lb]{\smash{{\SetFigFontNFSS{6}{7.2}{\familydefault}{\mddefault}{\updefault}{\color[rgb]{0,0,0}$v_1$}%
}}}}
\put(4501,-2086){\makebox(0,0)[lb]{\smash{{\SetFigFontNFSS{6}{7.2}{\familydefault}{\mddefault}{\updefault}{\color[rgb]{0,0,0}$v_2$}%
}}}}
\put(3376,-1861){\makebox(0,0)[lb]{\smash{{\SetFigFontNFSS{6}{7.2}{\familydefault}{\mddefault}{\updefault}{\color[rgb]{0,0,0}$v_5$}%
}}}}
\put(2926,-4336){\makebox(0,0)[lb]{\smash{{\SetFigFontNFSS{6}{7.2}{\familydefault}{\mddefault}{\updefault}{\color[rgb]{0,0,0}$v_3$}%
}}}}
\end{picture}%

\end{center}
and the curves $C_1, C_2, C_3, C_5$ are 
the toric curves corresponding to the 
rays through $v_1, v_2, v_3, v_5$ respectively.
In the sequel, we consider blow ups of 
$\FF_4$ and we will denote any proper 
transform of some $C_i$ again by $C_i$.

\begin{proposition}
\label{prop:f4bl3}
Let $X$ be the blow up of $\FF_4$ at 
points 
$c_0, c_\infty, c_1 \in \FF_4 \setminus C_1$, 
no two of them lying in a common fiber of 
$q \colon \FF_4 \to \PP_1$, and
let $C_4 \subseteq X$ 
be the exceptional divisor over
$c_0$.
Then $X$ is a smooth surface with 
\begin{eqnarray*}
\Cl(X) 
& \cong & 
\ZZ \mal w_1 \ \oplus \ \ldots \ \oplus \ \ZZ \mal  w_5,
\end{eqnarray*} 
where $w_i \in \Cl(X)$ denotes the 
class of the curve $C_i \subseteq X$.
Moreover, the Cox ring
of $X$ is the polynomial ring 
\begin{eqnarray*}
\mathcal{R}(X)
& = & 
\KK[T_1, \ldots, T_8]
/ \bangle {T_2T_4 + T_3T_6 + T_7T_8},
\end{eqnarray*}
and, with respect to the basis $(w_1, \ldots, w_5)$
of $\Cl(X)$, the degree of the generator
$T_i \in \mathcal{R}(X)$ is 
the $i$-th column of the matrix:
$$ 
\left[
\begin{array}{rrrrrrrr}
1 & 0 & 0 & 0 & 0 & 0 & -1 & 1
\\
0 & 1 & 0 & 0 & 0 & 1 & -2 & 3 
\\
0 & 0 & 1 & 0 & 0 & -1 & -1 & 1 
\\
0 & 0 & 0 & 1 & 0 & 1 & -1 & 2 
\\
0 & 0 & 0 & 0 & 1 & 0 & 1 & -1 
\end{array}
\right]
$$
\end{proposition}

\begin{proof}
We first reduce to the case that our three
points $c_0,c_1,c_\infty$ belong to the 
zero section $C_5 \subseteq \FF_4$ 
and within that $c_0$, $c_\infty$ 
are toric fixed points, 
whereas $c_1$ is the distinguished 
point of the nontrivial toric orbit;
note that $C_5$ is the closure of the 
toric orbit corresponding 
to $\cone(v_5)$.

We proceed in two steps.
First choose an automorphism of $\FF_4$
that moves $c_0,c_1$ and $c_\infty$ 
into the fibres over $0,1$ and $\infty$,
respectively.
Then construct a section $s \colon \PP_1 \to \FF_4$
mapping $0,1$ and $\infty$ to  $c_0,c_1$ and $c_\infty$,
and apply the automorphism $x \mapsto x - s(\pi(x))$,
where $\pi \colon \FF_4 \to \PP_1$ denotes the 
projection.

Blowing up the Hirzebruch surface $\FF_4$ at 
the points $c_0$ and $c_\infty$ gives 
a toric variety $X_0$; its fan looks as follows.
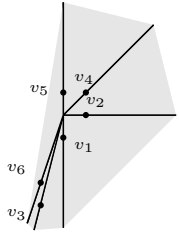
\begin{center}
\begin{picture}(0,0)%
\includegraphics{fb2blow.pstex}%
\end{picture}%
\setlength{\unitlength}{1243sp}%
\begingroup\makeatletter\ifx\SetFigFontNFSS\undefined%
\gdef\SetFigFontNFSS#1#2#3#4#5{%
  \reset@font\fontsize{#1}{#2pt}%
  \fontfamily{#3}\fontseries{#4}\fontshape{#5}%
  \selectfont}%
\fi\endgroup%
\begin{picture}(3423,4611)(2911,-4639)
\put(4276,-2986){\makebox(0,0)[lb]{\smash{{\SetFigFontNFSS{6}{7.2}{\familydefault}{\mddefault}{\updefault}{\color[rgb]{0,0,0}$v_1$}%
}}}}
\put(4501,-2086){\makebox(0,0)[lb]{\smash{{\SetFigFontNFSS{6}{7.2}{\familydefault}{\mddefault}{\updefault}{\color[rgb]{0,0,0}$v_2$}%
}}}}
\put(3376,-1861){\makebox(0,0)[lb]{\smash{{\SetFigFontNFSS{6}{7.2}{\familydefault}{\mddefault}{\updefault}{\color[rgb]{0,0,0}$v_5$}%
}}}}
\put(2926,-4336){\makebox(0,0)[lb]{\smash{{\SetFigFontNFSS{6}{7.2}{\familydefault}{\mddefault}{\updefault}{\color[rgb]{0,0,0}$v_3$}%
}}}}
\put(4276,-1636){\makebox(0,0)[lb]{\smash{{\SetFigFontNFSS{6}{7.2}{\familydefault}{\mddefault}{\updefault}{\color[rgb]{0,0,0}$v_4$}%
}}}}
\put(2926,-3436){\makebox(0,0)[lb]{\smash{{\SetFigFontNFSS{6}{7.2}{\familydefault}{\mddefault}{\updefault}{\color[rgb]{0,0,0}$v_6$}%
}}}}
\end{picture}%

\end{center}

\goodbreak

The matrix $P$ having $v_1, \ldots, v_6$ 
as its columns defines a surjection 
$\ZZ^6 \to \ZZ^2$ and hence an exact 
sequence.
As a matrix for the  
projection $\ZZ^6 \to \ZZ^4$ in the dual 
sequence, we may take
\begin{eqnarray*}
Q
& := & 
\left[
\begin{array}{rrrrrr}
1 & 0 & 0 & 0 & 1 & 0 
\\
0 & 1 & 0 & 0 & 3 & 1
\\
0 & 0 & 1 & 0 & 1 & -1
\\
0 & 0 & 0 & 1 & 2 & 1
\end{array}
\right]
\end{eqnarray*}

Assigning to $T_1, \ldots, T_6$ the columns 
$w_1, \ldots, w_6$ of $Q$ as their degrees, 
we obtain an action of the four torus 
$\TT^4 = (\KK^*)^4$ on $\KK^6$.
Note that $X_0$ is obtained as GIT-quotient 
of the set of $\TT^4$-semistable points 
associated to the 
$\TT^4$-linearization of the trivial bundle 
given by the weight $(3,11,2,10) \in \ZZ^4$;
in fact, by~\cite[Cor.~4.3]{Ha2},
we could take any weight from 
the relative interior of the moving cone
inside 
the moving cone 
\begin{eqnarray*}
\Mov(X)
& = & 
\bigcap_{i = 1}^6 \cone(w_j; \; j \ne i).
\end{eqnarray*}

In order to blow up the point $c_1 \in C_5$, 
we first embed $X_0$ in a suitable toric 
variety $Z_0$. 
Consider the polynomial ring 
$\KK[T_1, \ldots, T_7]$ with the 
$\ZZ^4$-grading given by 
$$
\deg(T_i) := w_i \text{ for } 1 \le i \le 6,
\qquad
\deg(T_7) := (0,1,0,1),
$$
and let $Q_0$ denote the matrix having these
degrees as its columns.
Then we have a surjection 
$\KK[T_1, \ldots, T_7] \to \KK[T_1, \ldots, T_6]$
of $\ZZ^4$-graded rings, defined by
$$ 
T_i \mapsto T_i \text{ for } 1 \le i \le 6,
\qquad
T_7 \mapsto T_2T_4 - T_3T_6.
$$
This gives a $\TT^4$-equivariant
embedding of $\b{X}_0 := \KK^6$ 
into $\b{Z}_0 := \KK^7$.
Note that the vanishing ideal of 
$\b{X}_0$ in $\b{Z}_0$ 
is generated by the polynomial
\begin{eqnarray*}
f_0 
& := &
T_7 -  T_2T_4 + T_3T_6.
\end{eqnarray*}

Consider the linearization of the trivial
bundle on $\b{Z}_0 = \KK^7$ given by 
the weight $(3,11,2,10) \in \ZZ^4$.
Then the corresponding set of semistable
points $\rq{Z}_0 \subseteq \b{Z}_0$
is an open toric subvariety.
The quotient $Z_0 := \rq{Z}_0 / \TT^4$ 
is a smooth projective toric variety;  
its fan $\Sigma_0$ has the 
columns $v_1', \ldots, v_7'$ 
of the matrix 
\begin{eqnarray*}
P_0
& = & 
\left[
\begin{array}{rrrrrrr}
0 & -1 &  1 & -1 & 0 & 1 & 0
\\
0 & -1 &  0 & -1 & 0 & 0 & 1
\\
-1 & -2 & -1 & -1 & 1 & 0 & -1
\end{array}
\right]
\end{eqnarray*}
corresponding to $Q_0$ as the 
primitive generators of its rays, 
and the 10 maximal cones of the fan $\Sigma_0$ 
are given as
$$ 
\cone(v_1',v_2',v_3'),\
\cone(v_1',v_2',v_7'),\
\cone(v_1',v_3',v_7'),\
\cone(v_2',v_3',v_6'),\
\cone(v_2',v_4',v_6'),
$$
$$
\cone(v_2',v_4',v_7'),\
\cone(v_3',v_6',v_7'),\
\cone(v_4',v_5',v_6'),\
\cone(v_4',v_5',v_7'),\
\cone(v_5',v_6',v_7').
$$

The closed embedding $\b{X}_0 \subseteq \b{Z}_0$
induces a closed embedding $X_0 \to Z_0$ of the quotient
spaces, and this is a neat embedding in the 
sense of~\cite{Ha2}, see~\cite[Prop.~3.14]{Ha2}.
By our choice of the embedding, the curve $C_5$ intersects 
the toric orbit corresponding to 
$\cone(v_5',v_7') \in \Sigma_0$ exactly in the point
$c_1$. 

Moreover, the polynomial $f_0$ is admissible,
the $g_{k_0}$-term is just $T_2T_4+T_3T_6$
and thus we may perform a toric ambient blow up
$Z_1 \to Z_0$
at the toric orbit closure corresponding to 
$\cone(v'_5,v'_7) \in \Sigma_0$.
Adding the column $v_5'+v_7' = (0,0,1)$
to the matrix $P_0$ gives the matrix $P_1$ 
describing the quotient presentation 
$\rq{Z}_1 \to Z_1$.

With the proper transform $X_1 \subseteq Z_1$ 
we obtain a modification $X_1 \to X_0$ 
of $X_0$ centered at the point $c_1 \in X_0$.
According to Proposition~\ref{recipe},
the Cox ring of $X_1$ is given as
\begin{eqnarray*}
\mathcal{R}(X_1)
& = & 
\KK[T_1, \ldots, T_8]
/ \bangle {T_7T_8 - T_2T_4 + T_3T_6},
\end{eqnarray*}
where the $\ZZ^5$-grading assigns to the 
generator $T_i$ the $i$-th column 
of the matrix $Q_1$ corresponding to $P_1$;
a direct computation shows that $Q_1$ 
is the matrix given in the assertion.

\goodbreak

We still have to check that $X_1 = X$ holds,
that means that the modification 
$X_1 \to X_0$ is indeed a blow up of 
the point $c_1 \in X_0$.
For this, note first that, 
according to~\cite[Cor.~4.13]{Ha2},
the variety $X_1$ inherits smoothness from its 
toric ambient variety $Z_1$.
Secondly, the exceptional curve 
over $c_1$ is smooth and rational,
and thus it must be a $(-1)$-curve.
\end{proof}

\section{K3-surfaces with a non-symplectic involution}
\label{sec-double}

We now take a closer look at (complex algebraic) 
K3-surfaces~$X$ admitting a non-symplectic involution,
i.e., an automorphism 
$\sigma \colon X \to X$ of order two such 
that $\sigma^* \omega_X = - \omega_X$ holds, 
where $\omega_X$ is a non zero holomorphic 
two form of $X$. 
Since $\Cl(X) = H^2(X,\ZZ) \cap \omega_X^{\perp}$ 
holds, and $\sigma$ is non-symplectic, 
one has 
$$
L^{\sigma}
\ := \
\{u \in H^2(X,\ZZ); \; \sigma^*(u)=u\}
\ \subseteq \ 
\Cl(X)
$$
for the fixed lattice.
The K3-surface $X$ is called 
{\em generic\/} if $\Cl(X)=L^{\sigma}$ 
holds; for fixed $\Cl(X)=L^{\sigma}$,
these K3-surfaces form a family of dimension
$20-\rk(L^{\sigma})$, see~\cite{n2}.
Our aim is to determine the Cox ring for 
the generic K3-surfaces with Picard number
$2\leq \varrho(X)\leq 5$, 
see Propositions~\ref{rhoX2} to~\ref{rhoX5},
and for those that are generic double 
covers of del Pezzo surfaces, 
see Proposition~\ref{doubledelp}.

For any K3-surface with a non-symplectic 
involution $\sigma \colon X \to X$, 
one has a quotient surface 
$Y := X / \bangle{\sigma}$ and the quotient 
map $\pi \colon X \to Y$.
We will use the following basic facts.

\begin{proposition}
\label{genk3double}
Let $X$ be a generic K3-surface with
a non-symplectic involution 
$\sigma \colon X \to X$.
Then the quotient map $\pi \colon X \to Y$
is a double cover and
\begin{enumerate}
\item 
if $\pi\colon X \to Y$ is unramified then 
the quotient surface $Y$ is an Enriques 
surface, 
\item
if $\pi\colon X \to Y$ is ramified, 
then $Y$ is a smooth rational surface and  
the following statements hold:
\begin{enumerate}
\item 
the branch divisor $B \in \WDiv(Y)$ 
of $\pi$ is smooth and, denoting by 
$K_Y$ the canonical 
divisor of $Y$, we have 
$$
\pi^*(B) 
\ = \ 
2\pi^{-1}(B),
\qquad 
B 
\ \sim \ 
-2 K_Y,
$$
\item 
the pullback $\pi^* \colon \Cl(Y) \to \Cl(X)$ 
is injective and $\pi^*(\Cl(Y))$ is 
of index $2^{n-1}$ in
$\Cl(X)$, where $n$ is the number of 
components of $B$.
\end{enumerate}
\end{enumerate}
\end{proposition}

\begin{proof}
The fact that $\pi \colon X \to Y$ is a 
double cover and~(i) as well as~(ii)
up to part~(b) are known, see~\cite[Lemma~1.2]{Z}.
In order to show part~(b) of~(ii) note first
that $\Cl(Y)$ is free, because $Y$ arises 
by blowing up points from $\PP_2$ or a 
Hirzebruch surface. 
According to~\cite[Ex.~1.7.6]{Fu}, 
we have
$$
\pi^*(\Cl_\QQ(Y)) 
\ = \
\Cl_\QQ(X)^\sigma,
\qquad
\pi_*\pi^*(\Cl(Y))
\ = \ 
2 \Cl(Y).
$$
Since $X$ is generic, the first
equation tells us that $\pi^*(\Cl(Y))$ 
is of finite index in~$\Cl(X)$.
The second one shows that $\pi^*$ 
is injective.
Moreover, by~\cite[Lemma~2.1]{BPV}, 
we have 
\begin{eqnarray*}
[\pi^*(\Cl(Y)):\Cl(X)]^2 
& = & 
\frac{\det \pi^*(\Cl(Y))}{\det \Cl(X)}.
\end{eqnarray*}
Since $\Cl(Y)$ is unimodular, see~\cite{BPV},
the numerator equals $2^{\varrho(Y)}$. 
By~\cite[Thm.~4.2.2]{n2},
the lattice $L^{\sigma}=\Cl(X)$ has determinant  
$2^l$ and
the difference $\varrho(Y)-l$ equals $2(n-1)$,
where $n$ is the number of connected 
components of the branch divisor. 
 \end{proof}

\begin{lemma}
\label{lem:doublecov}
Let $X$ be a generic K3-surface with
a non-symplectic involution 
such that the associated double cover
$\pi \colon X \to Y$ has 
branch divisor $B=C_1 + C_B$, 
where $C_1 \subseteq Y$ is a smooth 
rational curve and $C_B \subseteq Y$ 
is any irreducible curve.
\begin{enumerate}
\item
Let $w_1 \in \Cl(Y)$ be the class 
of $C_1 \subseteq Y$.
Then $(w_1,w_2, \ldots, w_r)$ is a basis of 
$\Cl(Y)$ if and only if
$(\pi^*(w_1)/2, \pi^*(w_2), \ldots, \pi^*(w_r))$ 
is a basis of $\Cl(X)$.
\item 
With respect to bases as in~(ii), 
the homomorphism
$\pi^* \colon \Cl(Y) \to \Cl(X)$ 
is given by the matrix
\begin{eqnarray*}
A 
& := &
{\tiny 
\left[
\begin{array}{cccc}
2 &   &        & 0 \\ 
  & 1 &        &   \\
  &   & \ddots &   \\
0 &   &        & 1
\end{array}
\right]
}.
 \end{eqnarray*}
\item 
Let $C \subseteq X$ be any smooth rational 
curve and let $w \in \Cl(X)$ be its class. 
Then precisely one of the following 
statements holds:
\begin{enumerate}
\item 
$\pi(C)$ is a component of $B$ 
and $\pi(C)^{2}=-4$,
\item
$w \in \pi^*(\Cl(Y))$ and $\pi(C)^2=-1$. 
\end{enumerate}
\end{enumerate}
\end{lemma}

\begin{proof}
Since $\pi^* \colon \Cl(Y) \to \Cl(X)$ 
is injective, 
$(w_1,\ldots,w_r)$ is a basis of 
$\Cl(Y)$ if and only if
$(\pi^*(w_1), \ldots, \pi^*(w_r))$ is
a basis of $\pi^*(\Cl(Y))$.
By Proposition~\ref{genk3double}~(ii),
we have $\pi^*(w_1)=2u_1$ with 
some $u_1 \in \Cl(X)$.
Moreover, also by
Proposition~\ref{genk3double}~(ii),
the pullback $\pi^*(\Cl(Y))$ is of index 
$2$ in $\Cl(X)$.
This gives~(i) and~(ii).

To prove~(iii), let $w \in \Cl(X)$ 
denote the class of $C \subseteq X$. 
The adjunction formula and Riemann-Roch
give $w^2=-2$ and $h^0(w)=1$. 
Since the elements of $\Cl(X)$ 
are fixed under the involution 
$\sigma \colon X \to X$, 
we can conclude $\sigma(C) = C$.
If $\sigma=\id$ holds on $C$, 
then $C$ is contained in the 
ramification divisor.
By Proposition~\ref{genk3double}~(ii),
we have $2C = \pi^{-1}(\pi(C))$, 
which implies $\pi(C)^2=-4$. 
If $\sigma \ne \id$ on $C$,
then the restriction $\pi \colon C \to \pi(C)$ 
is a double cover.
This implies 
$C=\pi^{-1}(\pi(C))$ and 
$\pi(C)^2 = 1/2 \mal C^2=-1$.  
\end{proof}

We are ready to describe the quotient surfaces 
of generic K3-surfaces with small Picard number.
In the sequel, we denote by $\Bl_k(Z)$ the blow 
up of a variety $Z$ in $k$ general points.
Moreover, we adopt the standard notation 
for integral lattices, see~\cite[Sec.~2, Chap.~I]{BPV},
and $L(k)$ denotes the lattice obtained 
from $L$ by multiplying the intersection matrix 
by $k$.

\begin{proposition}
\label{quot}
Let $X$ be a generic K3-surface $X$ with 
a non-symplectic involution 
and associated double cover 
$\pi \colon X \to Y$.
For $2 \le \varrho(X) \le 5$,
the table 

\medskip

\begin{center}
{\small
\begin{tabular}{c | l | l | l}
\label{lats}
$\varrho(X)$ & $\Cl(X)$ & Y  & B\\
\midrule
$2$  
& 
$U$,\ $U(2)$,\ $(2)\oplus A_1$ 
&  
$\FF_4,\ \FF_0,\ \Bl_1(\PP^2)$ 
&
$\PP_1 + C_{10}$, \ $C_9$, \ $C_9$
\\[5pt]
$3\leq k\leq 5$ 
& 
$U\oplus A_1^{k-2}$,\ $U(2)\oplus A_1^{k-2}$ 
&
$\Bl_{k-2}(\FF_4), \ \Bl_{k-2}(\FF_0)$  
&
$\PP_1 + C_{12-k}$, \ $C_{11-k}$
\end{tabular}
}
\end{center}

\medskip

\noindent
describes the intersection form of~$X$,
the quotient surface~$Y$
and the branch divisor $B$
of $\pi$,
where $C_g$ denotes a smooth irreducible curve of 
genus $g$. 
\end{proposition}

\begin{proof}
According to \cite[Sec.~4]{n2},
a lattice $L$ of rank at most five is the fixed 
lattice $L^{\sigma}$ of an involution 
$\sigma$ on a K3-surface if and only if it is 
an even lattice of signature $(1,k-1)$ which is 
$2$-elementary, i.e., satisfies 
$\Hom(L,\ZZ)/L=\ZZ_2^a$, where $2^a= \vert \det(L) \vert$.
Such lattices are classified up to isometries 
by three invariants: the rank $k$, the integer~$a$ 
and an invariant $\delta$ defined as 
\begin{eqnarray*}
\delta(L)
& = & 
\begin{cases}
0 & \text{if }  u^2 \in \ZZ \text{ for all } u \in \Hom(L,\ZZ),
\\
1 & \text{ otherwise}.
\end{cases}
\end{eqnarray*}
It is easy to check that the lattices in the table 
are the only $2$-elementary even lattices of 
signature $(1,k-1)$ with $2\leq k \leq 5$
since they cover all possible triples $(k,a, \delta)$,
see~\cite[Thm.~4.3.1]{n2} and also~\cite[Sec.~2.3]{AlNi}.

Now, suppose that the intersection form on $\Cl(X)$
is $U\oplus A_1^{k-2}$. Then it is known that 
there is an elliptic fibration 
$p \colon X\to \PP^1$ with a section $E$ and 
$k-2$ reducible fibers, see \cite[Lemma 3.1]{ko}.
In fact, if $e,f$ is the natural basis of 
$U$ and $v_1,\dots, v_{k-2}$ is an orthogonal 
basis of $A_1^{k-2}$, we can assume that 
the class of $E$ is $f-e$ and $v_i$ are 
represented by components of the reducible 
fibers not intersecting $E$.

By~\cite[Thm.~4.2.2]{n2} the ramification 
divisor of $\sigma$ is the disjoint union of 
a smooth irreducible curve of genus $12-k$ and 
a smooth irreducible rational curve. 
This implies that $C$ is transverse to the 
fibers of $p$, hence any fiber is preserved 
by $\sigma$ and the section $E$ is the rational 
curve in the ramification divisor. 

A basis of $\Cl(X)$ is given by 
$e,f-e, x_1,\dots, x_{k-2}$. 
It follows from Lemma~\ref{lem:doublecov}~(ii)
and~(iii) that the Picard lattice of $Y$ 
has intersection form
$$
\left(\begin{array}{cc}
0 & 1\\ 1 & -4\end{array}\right)
\ \oplus \ 
(-1)^{k-2}.
$$
Consequently, the classification of minimal rational 
surfaces yields that $Y$ is the blowing up of the 
Hirzebruch surface $\mathbb F_4$ at $k-2$ points.

Now assume that the intersection form on $\Cl(X)$ is
$U(2)\oplus A_1^{k-2}$.
Then, by~\cite[Thm.~4.2.2]{n2}, the ramification 
divisor has only one connected component and 
this is a smooth irreducible curve of genus $11-k$.
Thus, Proposition~\ref{genk3double}~(ii) gives
$\Cl(X)=\pi^*(\Cl(Y))$. 
It follows that the intersection form on
$\Cl(Y)$ is $U \oplus (-1)^{k-2}$. 
Hence, as before, we can conclude that $Y$ is 
the blowing up of $\mathbb F_0$ at $k-2$ points. 
  
Similarly, if the intersection form on $\Cl(X)$
is $(2)\oplus A_1$, then we obtain that the 
intersection form 
on $\Cl(Y)$ is $(1)\oplus (-1)$,
the ramificartion divisor is a smooth irreducible curve 
of genus $9$ and conclude that $Y$ is 
the blowing up of $\PP^2$ at one point.
\end{proof}

\goodbreak

\begin{proposition}
\label{f4double} 
Let $X$ be a generic K3-surface with a 
non-symplectic involution.
Suppose that 
\begin{itemize}
\item
the branch divisor 
of the associated double 
cover $\pi \colon X \to Y$ is of the 
form $B=C_1 + C_B$ with 
$C_1, C_B \subseteq Y$ irreducible 
and $C_1$ rational,
\item
the Cox ring of $Y$ 
is a polynomial ring $S = S'[t_1]$ with 
the canonical section $t_1$ of $C_1$
and a finitely generated $\CC$-algebra $S'$.
\end{itemize}
Moreover, denote by $f \in S'$ 
the canonical 
section of~$C_B$.
Then the Cox ring of $X$ is  
given as 
\begin{eqnarray*}
R 
& = &
\pi^*(S')[T_1,T_2] / \bangle{T_2^2-\pi^*(f)},
\end{eqnarray*}
with the $\Cl(X)$-grading defined by 
$\deg(\pi^*(g)) := \pi^*(\deg(g))$
for any homogeneous $g \in S'$ 
and 
$$
\deg(T_1) \ := \ \frac{\pi^*(w_1)}{2},
\qquad
\deg(T_2) \ := \ - \frac{\pi^*(2K_Y + w_1)}{2}.
$$
Moreover, the pullback homomorphism 
$\pi^* \colon S \to R$ of graded rings
is given on the grading groups by
$\ZZ^r \to \ZZ^r$, $w \mapsto Aw$  
and as a ring homomorphism by 
$$ 
t_1 \ \mapsto \ T_1^2, 
\qquad \qquad 
g  \ \mapsto \ \pi^*(g)
\text{ for any homogeneous } g \in S'.
$$
\end{proposition}

\begin{proof}
First note that by Proposition~\ref{prop:abcovfingen}, 
the Cox ring $R$ of $X$ inherits finite 
generation from the Cox ring $S$ of $Y$.
Consider the pullback group of $\Cl(Y)$ 
and the corresponding Veronese subalgebra 
$$
L \ := \ \pi^*(\Cl(Y)) \ \subseteq \ \Cl(X),
\qquad\qquad
R_{L} \ := \ \bigoplus_{w \in L} R_w.
$$

Write, for the moment $B = B_1+B_2$, and 
let $r$ and $b_i$ denote the canonical sections
of $\pi^{-1}(B)$ and $B_i$ respectively.
We claim that there is a commutative diagram
of finite ring homomorphisms
$$ 
\xymatrix{
R 
&
&
R_{L}
\ar[ll]
\ar[dl]
\\
& 
{\frac{R_{L}[u_1,u_2]}{\bangle{u_1^2-\pi^*b_1,u_2^2-\pi^*b_2,u_1u_2-r}}}
\ar[ul]^{\psi}
&
\\
& 
{\frac{S[T]}{\bangle{T^2-b_2}}}
\ar[u]_{\kappa}^{\cong}
\ar@/^2pc/[uul]
& 
}
$$
where, denoting by $r_1$ and $r_2$ the canonical 
sections of the reduced divisors
$\pi^{-1}(B_1)$ and $\pi^{-1}(B_2)$ respectively, 
the homomorphism $\psi$ is induced by $u_i \mapsto r_i$.

In this claim, everything is straightforward
except the definition of the isomorphism
$\kappa$.
By~\ref{prop:abcov} 
we know that $R_L$ is generated 
as a $\pi^*(S)$-module by 
$1$ and a section $s \in R_L$
satisfying $s^2 = \pi^*(b)$, where 
$b$ denotes the canonical section of $B$.
According to Lemma~\ref{lem:doublecov},
we may choose $s$ to be the canonical 
section $r$ of the ramification divisor 
$\pi^{-1}(B)$.
Thus, we obtain isomorphisms
\begin{eqnarray*}
& & 
R_{L}[u_1,u_2] / \bangle{u_1^2-b_1,u_2^2-b_2,u_1u_2-r}
\\
& \cong &
\pi^*(S)[y,u_1,u_2] / \bangle{y^2 - b,u_1^2-b_1,u_2^2-b_2,u_1u_2-y}
\\
& \cong &
\pi^*(S)[u_1,u_2] / \bangle{u_1^2-b_1,u_2^2-b_2}.
\end{eqnarray*}
Now we use our assumption $S = S'[t_1]$.
This enables us to define a ring homomorphism 
$$ 
\t{\kappa} \colon 
S[T] \ \to \ \pi^*(S)[u_1,u_2],
\qquad
S' \ni g \ \mapsto \ \pi^*(g) \in \pi^*(S'),
\quad
t_1 \mapsto u_1, 
\quad
T \mapsto u_2.
$$
It sends $t_1^2$ to $u_1^2$, which 
defines the same element in 
$\pi^*(S)[u_1,u_2] / \bangle{u_1^2-b_1,u_2^2-b_2}$
as $\pi^*(t_1)$. Consequently, $\t{\kappa}$
induces the desired isomorphism
$$ 
\kappa \colon 
S[T] / \bangle{T^2-b_2} 
\ \to \ 
\pi^*(S)[u_1,u_2] / \bangle{u_1^2-b_1,u_2^2-b_2}.
$$

The next step is to show that the homomorphism~$\psi$ 
of the above diagram is an isomorphism.
For this, it is enough to show that 
$S[T]/\bangle{T^2-b_2}$ is a normal ring.
Indeed, $R_L \to R$ is of degree two,
$$
R_L \ \to \ R_L[u_1,u_2]/\bangle{u_1^2-b_1,u_2^2-b_2,u_1u_2-r}
$$ 
is of degree at least two and thus $\psi$ is a 
finite morphism of degree one. If we know 
that $S[T]/\bangle{T^2-b_2}$ is normal, 
we can conclude that $\psi$ is an isomorphism.

In order to show that $S[T]/\bangle{T^2-b_2}$
is normal, note that $S$ can be made into a 
$\ZZ$-graded ring by assigning to each 
$\ZZ^r$-homogeneous element 
the $w_1$-component of its $\ZZ^r$-degree.
In particular, then $\deg(b_2)$ is odd.
Morever, $b_2 \in S$ is a prime element.
Thus, we can apply the result~\cite[p.~45]{ScSt}
and obtain that $S[T]/\bangle{T^2-b_2}$
is even factorial. In particular it is 
normal.

Having verified that $\psi$ is an isomorphism,
the commutative diagram tells us that the Cox 
ring $R$ of $X$ is isomorphic to 
$S[T]/\bangle{T^2-b_2}$.
Consequently, $R$ is the polynomial ring 
$\pi^*(S')[T_1,T_2]$ divided by the relation 
$T_2^2 - \pi^*(b_2)$, 
where $\pi^*(b_2)$ only depends on 
the first variable. 
The degrees of the generators $T_i$ are 
easily computed using Lemma~\ref{lem:doublecov}~(ii).
\end{proof}

\goodbreak

We are ready to compute the Cox rings of 
generic K3-surfaces $X$ admitting
a non-symplectic involution and 
satisfying $2 \le \varrho(X) \le 5$.
We will work with the curves 
$D_1, D_2 \subseteq \FF_0$ given by 
$$
D_1 \ := \ \{0\} \times \PP_1,
\qquad
\qquad
D_2 \ := \ \PP_1 \times \{0\}
,$$
 the curves $C_1, C_2 \subseteq \FF_4$ 
given by 
$$ 
C_1^2 \ = \ -4,
\qquad
C_2 \ := \ q^{-1}(0),
\qquad
C_3 := q^{-1}(\infty),
$$
where $q \colon \FF_4 \to \PP_1$ is the 
bundle projection, and the curves 
$E_1,E_2 \subseteq \Bl_1(\PP^2)$
with
$$
E_1^2 \ = \ 1,
\qquad
\qquad
E_2^2 \ =  \ -1.
$$
Moreover, on blow ups of the 
surfaces $\FF_0$ and $\FF_4$,
we denote the proper 
transforms of the curves $D_i$ 
and $C_j$ again by $D_i$ and $C_j$.

\begin{proposition}
\label{rhoX2}
Let $X$ be a generic K3-surface admitting a 
non-symplectic involution, and let 
$\pi \colon X \to Y$ be the associated double cover.
If $\varrho(X)=2$ holds, then the following 
cases can occur.
\begin{enumerate}
\item
We have $Y = \mathbb F_0$.
Then $\Cl(X) = \ZZ \mal \pi^*(w_1) \oplus \ZZ \mal \pi^*(w_2)$
holds, where $w_i  \in \Cl(Y)$ is the class 
of $D_i \in \WDiv(Y)$.
The Cox ring of $X$ is 
\begin{eqnarray*}
\mathcal{R}(X)
& = &
\CC[T_1,\ldots, T_5]/\bangle{T_5^2 - f}
\end{eqnarray*}
with a polynomial $f \in \CC[T_1, \ldots, T_4]$ 
and the degree of $T_i$ w.r.t. the above basis 
is the $i$-th column of the matrix
\begin{eqnarray*}
Q
& = & 
{\tiny
\left[
\begin{array}{rrrrr}
1 & 0 & 1 & 0 & 2
\\
0 & 1 & 0 & 1 & 2
\end{array}
\right]
}.
\end{eqnarray*}
\item
We have $Y = \FF_4$. Then $\Cl(X) =
\ZZ \mal \pi^*(w_1)/2 
\oplus 
\ZZ \mal \pi^*(w_2)
$
holds,
where $w_i  \in \Cl(Y)$ denotes the class 
of $C_i \in \WDiv(Y)$.
The Cox ring of $X$ is 
\begin{eqnarray*}
\mathcal{R}(X)
& = &
\CC[T_1,\ldots, T_5]/\bangle{T_5^2 - f}
\end{eqnarray*}
with a polynomial $f \in \CC[T_1^2, T_2, T_3, T_4]$ 
and, w.r.t. the above basis, the degree of $T_i$ is the 
$i$-th column of the matrix
\begin{eqnarray*}
Q
& = & 
{\tiny
\left[
\begin{array}{rrrrr}
1 & 0 & 2 & 0 & 3
\\
0 & 1 & 4 & 1 & 6
\end{array}
\right]
}.
\end{eqnarray*}
\item
We have $Y = Bl_1(\PP^2)$. Then $\Cl(X) =
\ZZ \mal \pi^*(w_1)
\oplus 
\ZZ \mal \pi^*(w_2)
$
holds,
where $w_i  \in \Cl(Y)$ denotes the class 
of $E_i \in \WDiv(Y)$.
The Cox ring of $X$ is 
\begin{eqnarray*}
\mathcal{R}(X)
& = &
\CC[T_1,\ldots, T_5]/\bangle{T_5^2 - f}
\end{eqnarray*}
with a polynomial $f \in \CC[T_1, T_2, T_3, T_4]$ 
and, w.r.t.~the above basis, the degree of $T_i$ is the 
$i$-th column of the matrix
\begin{eqnarray*}
Q
& = & 
{\tiny
\left[
\begin{array}{rrrrr}
1 & 0 & -1 & -1 & -1
\\
0 & 1 & 1 & 1 & 3
\end{array}
\right]
}.
\end{eqnarray*}
\end{enumerate}
\end{proposition}

\begin{proof}
First note that by 
Proposition~\ref{quot},
the surface $Y$ is one of 
the three types listed in the 
assertion.

If $Y  = \PP_1 \times \PP_1$ holds,
then $\Cl(Y) \cong \ZZ^2$ is generated
by the classes $w_1, w_2$ of $D_1, D_2$
and the Cox ring of $Y$ is given by
$$ 
\CC[T_1, \ldots, T_4],
\quad
\deg(T_1) = \deg(T_3) \ = w_1,
\quad
\deg(T_2) = \deg(T_4) \ = w_2,
$$
use e.g. Construction~\ref{constr:cox}.
Similarly, if $Y=\Bl_1(\PP^2)$, then 
$\Cl(Y) \cong \ZZ^2$ is generated
by the classes $w_1, w_2$ of $E_1, E_2$
and the Cox ring of $Y$ is given by
$$ 
\CC[T_1, \ldots, T_4],
\quad
\deg(T_1) = w_1,
\quad
\deg(T_2) = \deg(T_3) \ = w_1-w_2,
\quad
\deg(T_4)=w_2.
$$
In both cases, Proposition~\ref{quot}
tells us that the branch divisor $B \subseteq Y$ 
is irreducible. 
Propositions~\ref{genk3double}~(ii) 
and~\ref{prop:abcov} thus show that the 
Cox ring is as claimed in~(i) and~(iii).

If $Y = \FF_4$ holds, then
$\Cl(Y) \cong \ZZ^2$ is generated by 
clases $w_1,w_2$ of $C_1,C_2$
and the Cox ring of $Y$ is given by
$$ 
\CC[T_1, \ldots, T_4],
\quad
\deg(T_i) = w_1,
\quad
\deg(T_3) \ = w_1 + 4w_2,
\quad
\deg(T_2) = \deg(T_4) \ = w_2.
$$
This time, Lemma~\ref{lem:doublecov} and 
Proposition~\ref{f4double} show
that the Cox ring of $X$ is as claimed in~(ii).
\end{proof}

\goodbreak

\begin{proposition}
\label{rhoX3}
Let $X$ be a generic K3-surface admitting a 
non-symplectic involution, and let 
$\pi \colon X \to Y$ be the associated double cover.
If $\varrho(X)=3$ holds, then the following 
cases can occur.
\begin{enumerate}
\item
The surface $Y$ is the blow up of 
$\FF_0$ at the point $(0,0)$.
If $D_3 \subseteq Y$ denotes the 
exceptional curve, then
\begin{eqnarray*}
\Cl(X)
& = &
\ZZ \mal \pi^*(w_1) 
\ \oplus \
\ZZ \mal \pi^*(w_2)
\ \oplus \
\ZZ \mal \pi^*(w_3)
\end{eqnarray*}
holds, where $w_i  \in \Cl(Y)$ denotes the class 
of $D_i \in \WDiv(Y)$.
Moreover, the Cox ring of $X$ is given by 
\begin{eqnarray*}
\mathcal{R}(X)
& = &
\CC[T_1,\ldots, T_6]/\bangle{T_6^2 - f}
\end{eqnarray*}
with a polynomial $f \in \CC[T_1, T_2, \ldots, T_5]$ 
and, w.r.t. the above basis, the degree of $T_i$ is the 
$i$-th column of the matrix
\begin{eqnarray*}
Q
& = & 
{\tiny
\left[
\begin{array}{rrrrrr}
1 & 0 & 0 & 1 & 0 & 2
\\
0 & 1 & 0 & 0 & 1 & 2
\\
0 & 0 & 1 & 1 & 1 & 3
\end{array}
\right]
}.
\end{eqnarray*}
\item
The surface $Y$ is the blow up of $\FF_4$ 
at the point in $S_0 \cap q^{-1}(0)$,
where $S_0 \in \FF_4$ is the zero section. 
Then the  divisor class group of $X$ is  
\begin{eqnarray*}
\Cl(X)
& = &
\ZZ \mal \frac{\pi^*(w_1)}{2} 
\ \oplus \
\ZZ \mal \pi^*(w_2)
\ \oplus \
\ZZ \mal \pi^*(w_3),
\end{eqnarray*}
where $w_i  \in \Cl(Y)$ denotes the class 
of $C_i \in \WDiv(Y)$.
Moreover, the Cox ring of $X$ is given by 
\begin{eqnarray*}
\mathcal{R}(X)
& = &
\CC[T_1,\ldots, T_6]/\bangle{T_6^2 - f}
\end{eqnarray*}
with a polynomial $f \in \CC[T_1^2, T_2, \ldots, T_5]$ 
and, w.r.t. the above basis, the degree of $T_i$ is the 
$i$-th column of the matrix
\begin{eqnarray*}
Q
& = & 
{\tiny
\left[
\begin{array}{rrrrrr}
1 & 0 & 0 & 2 & 0 & 3
\\
0 & 1 & 0 & 1 & -1 & 1
\\
0 & 0 & 1 & 3 & 1 & 5
\end{array}
\right]
}.
\end{eqnarray*}
\end{enumerate}
\end{proposition}

\begin{proof}
The fact that
$Y$ is either $\PP_1 \times \PP_1$ 
blown up at a point $p$
or $\FF_4$ blown up at a point $p$
follows from Proposition~\ref{quot}. 
Moreover, the same Proposition 
yields that the branch divisor has 
one component in the first case 
and two components in the second one.
In both cases,
applying a suitable automorphism,
we may assume that the point $p$ 
to be blown up is as in the 
assertion. Then, in both cases,
the surface $Y$ is toric
and the computation of the Cox rings 
then goes the same way as in the 
preceding proposition.
\end{proof}

\begin{proposition}
\label{rhoX4}
Let $X$ be a generic K3-surface admitting a 
non-symplectic involution, and let 
$\pi \colon X \to Y$ be the associated double cover.
If $\varrho(X)=4$ holds, then the following 
cases can occur.
\begin{enumerate}
\item
The surface $Y$ is the blow up of $\FF_0$ 
at the points $(0,0)$ and $(\infty,\infty)$.
If $D_3,D_4 \subseteq Y$ are exceptional
curves corresponding to these points,
 then
\begin{eqnarray*}
\Cl(X)
& = &
\ZZ \mal \pi^*(w_1) 
\ \oplus \
\ldots
\ \oplus \
\ZZ \mal \pi^*(w_4)
\end{eqnarray*}
holds, where $w_i  \in \Cl(Y)$ denotes the class 
of $D_i \in \WDiv(Y)$.
Moreover, the Cox ring of $X$ is given by 
\begin{eqnarray*}
\mathcal{R}(X)
& = &
\CC[T_1,\ldots, T_7]/\bangle{T_7^2 - f}
\end{eqnarray*}
with a polynomial $f \in \CC[T_1, T_2, \ldots, T_6]$ 
and, w.r.t. the above basis, the degree of $T_i$ is the 
$i$-th column of the matrix
\begin{eqnarray*}
Q
& = & 
{\tiny
\left[
\begin{array}{rrrrrrr}
1 & 0 & 0 & 0 & 1 & 0 & 2
\\
0 & 1 & 0 & 0 & 0 & 1 & 2
\\
0 & 0 & 1 & 0 & 1 & 1 & 3
\\
0 & 0 & 0 & 1 & -1 & -1 & -1 
\end{array}
\right]
}.
\end{eqnarray*}
\item
The surface $Y$ is the blow up of $\FF_4$ 
at the two points $p_1 \in S_0 \cap q^{-1}(0)$
and $p_2 \in S_0 \cap q^{-1}(\infty)$,
where $S_0 \in \FF_4$ is the zero section. 
We have 
\begin{eqnarray*}
\Cl(X)
& = &
\ZZ \mal \frac{\pi^*(w_1)}{2} 
\ \oplus \
\ZZ \mal \pi^*(w_2)
\ \oplus \
\ZZ \mal \pi^*(w_3)
\ \oplus \
\ZZ \mal \pi^*(w_4),
\end{eqnarray*}
where $w_i  \in \Cl(Y)$ is the class 
of $C_i \in \WDiv(Y)$ and $C_4 \subseteq Y$ is the 
exceptional curve over $p_1 \in \FF_4$.
The Cox ring of $X$ is given by 
\begin{eqnarray*}
\mathcal{R}(X)
& = &
\CC[T_1,\ldots, T_7]/\bangle{T_7^2 - f}
\end{eqnarray*}
with a polynomial $f \in \CC[T_1^2, T_2, \ldots, T_6]$ 
and, w.r.t. the above basis, the degree of $T_i$ is the 
$i$-th column of the matrix
\begin{eqnarray*}
Q
& = & 
{\tiny
\left[
\begin{array}{rrrrrrr}
1 & 0 & 0 & 0 & 2 & 0 & 3
\\
0 & 1 & 0 & 0 & 3 & 1 & 5
\\
0 & 0 & 1 & 0 & 1 & -1 & 1
\\
0 & 0 & 0 & 1 & 2 & 1 & 4
\end{array}
\right]
}.
\end{eqnarray*}
\end{enumerate}
\end{proposition}

\begin{proof}
Again, Proposition~\ref{quot}
tells us that
$Y$ is either $\PP_1 \times \PP_1$ 
blown up at two points $p,q$
or $\FF_4$ blown up at two points 
$p,q$ and that the branch divisor has 
one component in the first case 
and two components in the second one.
In both cases, we may 
apply a suitable automorphism,
and achieve that the points $p,q$ 
to be blown up are as in the 
assertion. Thus, again,
the surface $Y$ is toric
and that computation of the Cox 
rings proceeds as before.
\end{proof}

\goodbreak

\begin{proposition}
\label{rhoX5}
Let $X$ be a generic K3-surface admitting a 
non-symplectic involution, and let 
$X \to Y$ be the associated double cover.
If $\varrho(X)=5$ holds, then the folllowing 
cases can occur.
\begin{enumerate}
\item
The surface $Y$ is the blow up of $\FF_0$ 
at three general points.
Then the Cox ring of $X$ is
\begin{eqnarray*}
\mathcal{R}(X)
& = &
\CC[T_1,\ldots, T_{11}]/\bangle{f_1, \ldots, f_5, T_{11}^2 - g}
\end{eqnarray*}
where $f_1, \ldots, f_5$ are the Pl\"ucker relations 
in the variables $T_1, \ldots, T_{10}$, i.e., we have 
$$ 
f_1 \ = \ T_2T_5 - T_3T_6 + T_4T_7,
\qquad
f_2 \ = \ T_1T_5 - T_3T_8 + T_4T_9,
$$
$$
f_3 \ = \ T_1T_6 - T_2T_8 + T_4T_{10},
\qquad
f_4 = T_1T_7 - T_2T_9 + T_3T_{10},
$$
$$
f_5 \ = \ T_5T_{10} - T_6T_9 + T_7T_8
$$
and 
$g \in \CC[T_1,\ldots, T_{10}]$ is a prime polynomial.
The degree of $T_i \in \mathcal{R}(X)$ is
the $i$-th column of
\begin{eqnarray*}
Q & = & 
{\tiny
\left[
\begin{array}{rrrrrrrrrrr}
0 & 0 & 0 & 0 & 1 & 1 & 1 & 1 & 1 & 1 &    -3 
\\
1 & 0 & 0 & 0 & -1 & -1 & -1 & 0 & 0 & 0 &  1
\\
0 & 1 & 0 & 0 & -1 & 0 & 0 & -1 & -1 & 0 &  1
\\
0 & 0 & 1 & 0 & 0 & -1 & 0 & -1  & 0 & -1 & 1
\\
0 & 0 & 0 & 1 & 0 & 0 & -1 & 0 & -1 & -1 & 1
\end{array}
\right]
}
\end{eqnarray*}
\item
The surface $Y$ is the blow up of $\FF_4$ 
at three general points.
Then the Cox ring of $X$ is
\begin{eqnarray*}
\mathcal{R}(X)
& = &
\CC[T_1,\ldots, T_9]/\bangle{T_2T_5 + T_4T_6+T_7T_8, \, T_9^2 - f}
\end{eqnarray*}
where $f \in \CC[T_1, \ldots, T_8]$ is a 
prime polynomial and
the degree of $T_i \in \mathcal{R}(X)$ is
the $i$-th column of
\begin{eqnarray*}
Q 
& = & 
{\tiny
\left[
\begin{array}{rrrrrrrrr}
1 & 0 & 0 & 0 & 0 & 0 & -2 & 2 & 1
\\
0 & 1 & 0 & 0 & 0 & 1 & -2 & 3  & 4
\\
0 & 0 & 1 & 0 & 0 & -1 & -1 & 1  & 0
\\
0 & 0 & 0 & 1 & 0 & 1 & -1 & 2 & 4
\\
0 & 0 & 0 & 0 & 1 & 0 & 1 & -1 & 1
\end{array}
\right]
}
\end{eqnarray*}
\end{enumerate}
\end{proposition}

\begin{proof}
The facts that only (i) and (ii) are 
possible and that the 
branch divisor has one component
in~(i) and two in~(ii) follow
from Proposition~\ref{quot}.

If we are in the situation~(i), 
then $Y$ is the blow up of $\PP_2$ 
at four general points
and hence is a del Pezzo surface.
Its Cox ring is the ring of 
$(3 \times 3)$-minors of a generic 
$(3 \times 5)$-matrix, 
see~\cite[Prop.~4.1]{BaPo}.
Moreover, by Proposition~\ref{genk3double}~(ii),
the pullback $\pi^* \colon \Cl(Y) \to \Cl(X)$ 
is an isomorphism.
Thus, taking the same basis of $\Cl(Y)$ 
as in the proof of~\cite[Prop.~4.1]{BaPo},
the assertion follows from 
Propositions~\ref{prop:abcov} 
and~\ref{genk3double}~(ii).

If we are in situation~(ii), then the assertion 
is a direct consequence of 
Propositions~\ref{prop:f4bl3}, \ref{f4double}
and Lemma~\ref{lem:doublecov}.
Note that the canonical class of $Y$ can
be determined according to~\cite[Prop.~8.5]{BeHa2}
as the degree of the relation minus the sum
of the degrees of the generators of the 
Cox ring of $Y$.
\end{proof}

If $X$ is a generic K3-surface with a non-symplectic 
involution such that the associated double cover 
has an irreducible branch divisor, then we can 
proceed the computation of Cox rings as follows.

\begin{proposition} 
\label{doubledelp}
Let $X$ be a generic K3-surface 
with a non-symplectic involution,
associated double cover 
$\pi \colon X \to Y$
and intersection form
$U(2)\oplus A_1^{k-2}$, 
where $5 \le k \le 9$.
Then $Y$ is a Del Pezzo surface of 
Picard number $k$
and
\begin{enumerate}
\item
the Cox ring $\mathcal{R}(X)$ is generated by the
pull-backs of the $(-1)$-curves of $Y$, 
the section $T$ defining the ramification divisor
and, for $k =9$, 
the pull-back of an irreducible section
of $H^0(Y,- K_Y)$,
\item
the ideal of relations of $\mathcal{R}(X)$ 
is generated by quadratic
relations of degree $\pi^*(D)$, 
where $D^2=0$ and $D \mal K_Y=-2$,
and the relation $T^2-f$ in degree 
$-2\pi^*(K_Y)$, where $f$ is the pullback
of the canonical section of the branch divisor.
\end{enumerate}
\end{proposition}

\begin{proof}
As in the proof of Proposition~\ref{quot},
we use~\cite[Thm.~4.2.2]{n2} to see that 
the ramification divisor of $\pi \colon X \to Y$ 
is irreducible.
Then Proposition~\ref{genk3double}~(ii) yields
$\Cl(X)=\pi^*(\Cl(Y))$. 
It follows that the intersection form on
$\Cl(Y)$ is $U \oplus (-1)^{k-2}$.
Consequently, $Y$ is the blow up of 
$\FF_0$ at $k-2$ general points and
hence is a del Pezzo surface.

It is known that ${\mathcal R}(Y)$ is
generated by all the $(-1)$-curves of $Y$ plus, if $k=9$,
an irreducible section of $H^0(Y,-K_Y)$,
see~\cite[Thm.~3.2]{BaPo}.
The ideal of relations of $\mathcal{R}(Y)$ 
is generated by quadratic
relations of degree $D$, where $D$ is a conic bundle, 
i.e., we have $D^2=0$ and
$D \mal K_Y = -2$,
see~\cite[Lemmas~2.2, 2.3 and~2.4]{lv}.
Thus the statement follows from
Proposition~\ref{prop:abcov}.
\end{proof}

\end{document}